\documentclass{siamart1116}
\usepackage[T1]{fontenc}
\usepackage[utf8]{inputenc}
\numberwithin{theorem}{section}

\usepackage{lipsum}
\usepackage{amsfonts}
\usepackage{graphicx}
\usepackage{epstopdf}
\usepackage{algorithmic}

\usepackage[T1]{fontenc}
\usepackage[UKenglish]{babel}%
\usepackage{amssymb, amsmath}
\usepackage{lmodern}
\usepackage{dsfont}
\usepackage{enumerate}
\usepackage[all]{xy}
\usepackage{mathtools}
\usepackage{appendix}
\usepackage{pgf,tikz}
\usepackage{tikz-cd}
\usetikzlibrary{arrows}
\usepackage{centernot}
\usepackage{verbatim}
\usepackage[normalem]{ulem}

\ifpdf
  \DeclareGraphicsExtensions{.eps,.pdf,.png,.jpg}
\else
  \DeclareGraphicsExtensions{.eps}
\fi

\usepackage{amsopn}
\DeclareMathOperator{\diag}{diag}
\newcommand{\sep}{,\,}

\DeclareFontFamily{U}{mathx}{\hyphenchar\font45}
\DeclareFontShape{U}{mathx}{m}{n}{
      <5> <6> <7> <8> <9> <10>
      <10.95> <12> <14.4> <17.28> <20.74> <24.88>
      mathx10
      }{}
\DeclareSymbolFont{mathx}{U}{mathx}{m}{n}
\DeclareFontSubstitution{U}{mathx}{m}{n}
\DeclareMathAccent{\widecheck}{0}{mathx}{"71}
\DeclareMathAccent{\wideparen}{0}{mathx}{"75}

\usepackage{mathrsfs}  

\newcommand{\grad}{\nabla}
\newcommand{\vc}[1]{\mathbf{#1}}

\newcommand{\N}{\mathbb{N}}
\newcommand{\R}{\mathbb{R}}

\renewcommand{\S}{\mathbb{S}}

\newcommand{\ps}[2]{\left\langle #1 , #2 \right\rangle}
\newcommand{\norm}[1]{\| #1 \|}
\newcommand{\nnorm}[1]{{\vert\kern-0.25ex\vert\kern-0.25ex\vert #1 
\vert\kern-0.25ex\vert\kern-0.25ex\vert}}
\renewcommand{\bar}[1]{\overline{#1}}
\newcommand{\sign}{\textup{sign}}

\renewcommand{\epsilon}{\varepsilon}

\newcommand{\bal}{\boldsymbol{\alpha}}
\newcommand{\bdel}{\boldsymbol{\delta}}
\newcommand{\bbe}{\boldsymbol{\beta}}

\newcommand{\bt}{\boldsymbol{\theta}}

\newcommand{\bsig}{\boldsymbol{\sigma}}

\newcommand{\blam}{\boldsymbol{\lambda}}
\newcommand{\bnu}{\boldsymbol{\nu}}

\newcommand{\bb}{\vc b}

\newcommand{\be}{\vc e}
\newcommand{\bp}{\vc p}

\newcommand{\bx}{\vc x}

\newcommand{\by}{\vc y}

\newcommand{\bz}{\vc z}

\newcommand{\bv}{\vc v}

\newcommand{\bu}{\vc u}

\newcommand{\bw}{\vc w}

\newcommand{\and}{\quad\text{and}\quad}

\newcommand{\E}{\mathcal{E}}

\newcommand{\kone}{\mathcal{K}}

\newcommand{\lek}{\leq_{\kone}}
\newcommand{\lekk}{\lneq_{\kone}}
\newcommand{\lekkk}{<_{\kone}}

\newcommand{\T}{\mathcal{T}}

\def\G{\mathcal{G}}

\newcommand{\krog}{\otimes}

\newcommand{\bj}{\vc j}

\newcommand{\I}{\mathcal I}
\newcommand{\saufzero}{\setminus\{0\}}

\newcommand{\ones}{\mathbf{1}}

\newcommand{\sauf}{\setminus}

\usepackage{soul} 

\newcommand{\cwlt}{\widehat{\operatorname{cw}}}
\newcommand{\cwut}{\widecheck{\operatorname{cw}}}

\newsiamthm{thm}{Theorem}
\newsiamthm{lem}{Lemma}
\newsiamthm{prop}{Proposition}
\newsiamthm{cor}{Corollary}
\newsiamremark{rmq}{Remark}
\newsiamremark{ex}{Example}
\newsiamremark{defi}{Definition}
\newsiamremark{prb}{Problem}

\newcommand{\TheTitle}{A Unifying Perron-Frobenius Theorem for Nonnegative Tensors via Multi-homogeneous Maps} 
\newcommand{\TheAuthors}{A. Gautier, F. Tudisco, and M. Hein}

\headers{Perron-Frobenius Theorem for Nonnegative Tensors}{\TheAuthors}

\title{{\TheTitle}\thanks{Author's accepted version: this is the peer-reviewed version of this manuscript, which is now   published on SIAM Journal on Matrix Analysis and Applications \url{https://doi.org/10.1137/18M1165049}.
\funding{This work has been funded by the ERC starting grant ``NOLEPRO'', no.\ 307793. The work of F.T. was funded by the European Union's Horizon 2020 research and innovation programme under the MarieSk\l odowska-Curie individual fellowship ``MAGNET'' grant agreement no.\ 744014.}}}

\author{
  Antoine Gautier\thanks{Department of Mathematics and Computer Science, Saarland University, 66041 Saarbr\"{u}cken, Germany
    (\email{ag@cs.uni-saarland.de},\email{hein@math.uni-sb.de}).}
  \and
 Francesco Tudisco\thanks{Department of Mathematics and Statistics, University of Strathclyde, G11XH Glasgow, UK (\email{f.tudisco@strath.ac.uk}).}
  \and
 Matthias Hein\footnotemark[2]
}

\usepackage{mathtools}
\newcommand{\dimensionalpartition}{shape partition}
\newcommand{\Dimensionalpartition}{Shape partition}

\usepackage{rotating}
\newcommand{\verteq}[0]{\begin{turn}{90}$=\,\,\,$\end{turn}}

{\color{magenta} }%
{\color{black} }

\begin{document}

\maketitle

\begin{abstract} 
We introduce the concept of \dimensionalpartition{} of a tensor and formulate a general tensor eigenvalue problem that includes all previously studied eigenvalue problems as special cases. We formulate irreducibility and symmetry properties of a nonnegative tensor $T$ in terms of the associated \dimensionalpartition{}. We recast the eigenvalue problem for $T$ as a fixed point problem on a suitable product of projective spaces. This allows us to use the theory of multi-homogeneous order-preserving maps to derive a new and  unifying Perron-Frobenius theorem for nonnegative tensors which either implies earlier results of this kind or improves them, as weaker assumptions are required. We introduce a general power method for the computation of the dominant tensor eigenpair, and provide a detailed convergence analysis. 
\end{abstract}


\section{Introduction}
Tensor eigenvalue problems have gained considerable attention in recent years as they arise in a number of relevant applications, such as best rank-one approximation in data analysis \cite{de2000best,Qi2007}, higher-order Markov chains \cite{NQZ}, solid mechanics and
the entanglement problem in quantum physics \cite{Chang_rect_eig,Qi_rect}, multi-layer network analysis \cite{arrigo2019multi,frafra}.   
A number of contributions have addressed relevant issues both form the theoretical and numerical point of view. 
The multi-dimensional nature of tensors naturally gives rise to a variety of eigenvalue problems. In fact, the classical
eigenvalue and singular value problems for a matrix can be generalized to the tensor setting following different constructions which lead to different notions of
eigenvalues and singular values for tensors, all of them reducing to the standard matrix case when the tensor is assumed to be of order two. Moreover, the extension of the power method to the  tensor setting, including certain shifted variants, is the best known method for the computation of tensor eigenpairs \cite{kolda2011shifted,regalia2000higher}. 

When the tensor has nonnegative entries, many authors have worked on tensor generalizations of the Perron-Frobenius theorem for matrices \cite{Boyd,Chang,Fried,Lim,NQZ}. In this setting, existence, uniqueness and maximality of positive eigenpairs of the tensor are discussed, in terms of certain irreducibility assumptions. Moreover, as for the matrix case,  Perron-Frobenius type results allow to address the global convergence of the power method for tensors with nonnegative entries \cite{Boyd,Chang_rect_eig,us,NQZ}.  

However, all the contributions that have appeared so far address particular cases of tensor spectral problems individually. 
In this work we  formulate a general tensor spectral problem which includes known formulations as special cases. Moreover, we prove a new Perron-Frobenius theorem for the general tensor eigenvalue problem which allows  to retrieve previous results as particular cases and, often, allows to significantly weaken the assumptions previously made. In addition, we prove the global convergence of a nonlinear version of the  power method that allows to compute the dominant eigenpair for general tensor eigenvalues, under mild assumptions on the tensor and with an explicit upper bound on the convergence rate.  

We first illustrate the discussed spectral problems for the special case of an entrywise nonnegative square tensor of order three, $T=(T_{i,j,k})\in \R^{N\times N \times N}$. 
We omit here most of the details, which are instead carefully discussed in the next sections. Let $f_T\colon\R^N\times\R^N\times \R^N\to \R$ denote the multilinear form induced by $T$, 
\begin{equation*}
f_T(\bx,\by,\bz)=\sum_{i,j,k=1}^N T_{i,j,k}x_{i}y_jz_k \qquad \forall \bx,\by,\bz\in\R^n\, , 
\end{equation*}
and, for  $p,q,r\in(1,\infty)$, consider the  following Rayleigh quotients: 
\begin{equation}\label{defRR}
\Phi^{1}(\bx)= \frac{f_T(\bx,\bx,\bx)}{\norm{\bx}_p^3}, \quad\Phi^2(\bx,\by)=  \frac{f_T(\bx,\by,\by)}{\norm{\bx}_p\norm{\by}_q^2}, \quad\Phi^{3}(\bx,\by,\bz)= \frac{f_T(\bx,\by,\bz)}{\norm{\bx}_p\norm{\by}_q\norm{\bz}_r}.
\end{equation}

Note that, since the tensor is nonnegative and has odd order,  the maximum of $\Phi^i$ provides a notion of norm of $T$, for $i=1,2,3$. Furthermore, note that $\Phi^1,\Phi^2$ and $\Phi^3$ lead naturally to the 
definition of $\ell^p$-eigenvectors, $\ell^{p,q}$-singular vectors and $\ell^{p,q,r}$-singular vectors of the tensor $T$ \cite{Lim}. These are respectively defined as the solutions of the following spectral equations
\begin{equation}\label{eqsys}
\T_1(\bx,\bx,\bx)=\lambda \psi_p(\bx), \quad \begin{cases} \T_1(\bx,\by,\by)=\lambda \psi_p(\bx),\\ \T_2(\bx,\by,\by)=\lambda\psi_q(\by),\end{cases} \quad  \begin{cases} \T_1(\bx,\by,\bz)=\lambda \psi_p(\bx),\\ 
\T_2(\bx,\by,\bz)=\lambda\psi_q(\by),\\ 
\T_3(\bx,\by,\bz)=\lambda\psi_r(\bz),\end{cases}
\end{equation}
where $\psi_p(\bx)=\frac{1}{p}\grad \norm{\bx}_p^{p}=(|x_1|^{p-2}x_1,\dots,|x_N|^{p-2}x_N)$ and, for $i=1,2,3$, the mapping $\T_i(\bx,\by,\bz)$ is the gradient of $\bx_i \mapsto f_T(\bx_1,\bx_2,\bx_3)$.

It is well known that the singular values of a matrix always admit a variational characterization, whereas the same holds for eigenvalues only if the matrix is symmetric. A similar situation occurs for tensors, where suitable symmetry assumptions on $T$ are required in order to relate  the critical points of the Rayleigh quotients in \eqref{defRR} with the solutions of the spectral equations in \eqref{eqsys}:
 If $T$ is super symmetric, i.e.\ the entries of $T$ are invariant under any permutation of its indices, then $\grad f_T(\bx,\bx,\bx)=3\T_1(\bx,\bx,\bx)$ and so the correspondence between the critical points of $\Phi^1$ and the solutions to $\T_1(\bx,\bx,\bx)=\lambda \psi_p(\bx)$ is clear. If $T$ is partially symmetric with respect to its second and third indices, i.e. $T_{i,j,k}=T_{i,k,j}$ for every $i,j,k\in[N]=\{1,\ldots,N\}$, then $\grad_{\bx} f_T(\bx,\by,\by)=\T_1(\bx,\by,\by)$ and $\grad_{\by} f_T(\bx,\by,\by)=2\T_2(\bx,\by,\by)$ and, again, it is not difficult to observe that the critical points of $\Phi^2$ coincide with the solutions to the second system in \eqref{eqsys}. Finally, the third system in \eqref{eqsys} always characterizes the critical points of $\Phi^3$ as $\grad f_T = (\T_1,\T_2,\T_3)$. This latter case is the analogue of the singular value problem for matrices.  

 In the case where $T$ does not have such symmetries, the critical points of $\Phi^1$ and $\Phi^2$ are solutions to spectral systems analogous to those in 
 \eqref{eqsys} but where the mapping $\T_i$ is the gradient of $\bx_i \mapsto f_{S}(\bx_1,\bx_2,\bx_3)$ and $S\in\R^{N\times N \times N}_+$ is a symmetrized version of $T$ whose construction depends on the considered problem. We discuss this property in detail in Section \ref{tnormsec}. Note that this phenomenon is, again, aligned with the matrix case. In fact, the quadratic form associated to a matrix $M$ always coincides with the form associated with the symmetric matrix $(M^\top+M)/2$. 
Now, as $T$ has nonnegative entries, the triangle inequality implies that $|\Phi^1(\bx)|\leq \Phi^1(|\bx|)$ for every $\bx\in\R^N\saufzero$, where the absolute value is taken component wise. In particular, this implies that the maximum of $\Phi^1$ is attained in the nonnegative orthant $\R^N_+=\{\bx\in\R^N\mid x_i\geq 0, \ \forall i \in[N] \}$. Similar simple arguments show that the maxima of $\Phi^2$ and $\Phi^3$ are attained on nonnegative vectors as well. There is a vast literature on the study of the solutions to the systems in \eqref{eqsys} in the particular setting where $T$ is nonnegative. We refer to it as the Perron-Frobenius theory for nonnegative tensors \cite{NLA:NLA1902}. 
Typical results of the latter theory provide conditions on the parameters $p,q,r$ and on the irreducibility structure of the tensor $T$ to ensure the existence, the uniqueness and the maximality of positive solutions to particular cases of the systems in \eqref{eqsys}. These results come together with a number of Collatz-Wielandt type characterizations and with the convergence analysis of particular tensor versions of the power method. See for e.g.  \cite{Chang,Fried,us,Hu2014,linlks} and references therein.

In this paper, we address tensors of any order and propose a framework that allows us to unify the study of all spectral equations of the type shown in 
\eqref{eqsys} and to prove a general Perron-Frobenius theorem which either improves the known results mentioned above or includes them as special cases. In particular, we give new conditions for the existence, uniqueness and maximality of positive eigenpairs for an ample class of tensor spectral equations, we prove new characterizations for the maximal eigenvalue and 
 we discuss the convergence of the power method including explicit rates of convergence. This is done by introducing 
a parametrization, which we call \textit{\dimensionalpartition{}}, 
so that the three problems discussed in \eqref{eqsys} can be recovered with a suitable choice of the parameters. 
Moreover, shape partitions allow us to introduce general definitions of weak and strong irreducibility, which both reduce to existing counter parts for suitable choices of the partition.
We discuss in detail the relationship between different types of irreducible 
nonnegative tensors and we show how they are related for 
 different spectral equations. 

 A particular contribution of this paper is that we reformulate the tensor spectral problems in terms of suitable multi-homogeneous maps and the associated fixed points on a product of projective spaces. Thus, based on our results in \cite{mhpfpaper}, we show that most of the tensor spectral problems correspond to a multi-homogeneous  mapping that is contractive with respect to a suitably defined projective metric. This relatively simple  observation turns out to be very relevant as it allows to systematically weaken the assumptions made in the   Perron-Frobenius literature for nonnegative tensors so far. The paper is written in a self-contained manner. However, for the proofs we rely heavily on our results from \cite{mhpfpaper}.
\section{Preliminaries}\label{intro}
In this section we fix the main notation and definitions that are required to formulate the Rayleigh quotients in \eqref{defRR} and the associated spectral problems in a unified fashion for the general case of a tensor of any order and with possibly different dimensions. 

Let $\R^{N_1\times\ldots\times N_m}_+$ be the set of entrywise nonnegative tensors in $\R^{N_1\times\ldots\times N_m}$. Let $T\in\R^{N_1\times\ldots\times N_m}_+$ and define the induced multilinear form $f_T\colon \R^{N_1}\times\ldots\times\R^{N_m}\to\R$ as 
\begin{equation*}
f_T(\bz_1,\ldots,\bz_m) = \sum_{j_1\in[N_1],\,\ldots,\,j_m\in[N_m]} T_{j_1,\ldots,j_m}z_{1,j_1}z_{2,j_2}\cdots z_{m,j_m},
\end{equation*}
where $[N_i]=\{1,\ldots,N_i\}$ for all $i$. Furthermore, let us consider the gradient of $f_T$, that is let $\T=(\T_1,\ldots,\T_m)$ with $\T_i=(\T_{i,1},\ldots,\T_{i,N_i})$ and $\T_{i,j_i}\colon \R^{N_1}\times\ldots\times\R^{N_m}\to\R$ defined as
\begin{equation*}
\T_{i,j_i}(\bz_1,\ldots,\bz_m) = \sum_{\substack{j_1\in[N_1],\ldots,j_{i-1}\in[N_{i-1}]\\ j_{i+1}\in[N_{i+1}],\ldots,j_m\in[N_m]}} T_{j_1,\ldots,j_m}z_{1,j_1}\cdots z_{i-1,j_{i-1}}z_{i+1,j_{i+1}}\cdots z_{j_m,m}
\end{equation*}

As for the case of a square tensor of order three, described in the previous section, several  Rayleigh quotients and spectral equations can be associated to $T$. For instance, we have now up to $m$ different choices of the norms in the denominator of \eqref{defRR}. Moreover, various choices for the numerator are possible,  depending on how one partitions the dimensions of $\R^{N_1}\times \cdots \times\R^{N_m}$. In order to formalize these properties for a general tensor $T$, we introduce here the concept of  \textit{\dimensionalpartition{}}.

\begin{defi}[\Dimensionalpartition]\label{defsig}
We say that $\bsig$ is a \dimensionalpartition{} of $T\in\R^{N_1\times \ldots \times N_m}$ if $\bsig = \{\sigma_i\}_{i=1}^d$ is a partition of $[m]$, i.e.\ $\cup_{i=1}^d \sigma_i = [m]$ and $\sigma_i\cap \sigma_j=\emptyset$ for $i\neq j$, such that for every $i\in[d]$ and $j,j'\in\sigma_i$, it holds $N_j=N_{j'}$. Moreover, we always assume that:
\begin{enumerate}[$\quad (a)$]
\item For every $i\in [d-1]$ and $j\in\sigma_i,k\in\sigma_{i+1}$ it holds $j\leq k$.\label{dimpart3}
\item If $d>1$, then $|\sigma_i|\leq |\sigma_{i+1}|$ for every $i\in[d-1]$.\label{dimpart4}
\end{enumerate}
\end{defi}
Observe that the conditions \eqref{dimpart3} and \eqref{dimpart4} in the above definition are not restrictive. Indeed, if $\bsig = \{\sigma_i\}_{i=1}^d$ is a partition of $[m]$ such that  $N_j=N_{j'}$ for every $j,j'\in\sigma_i$ and  $i\in[d]$, then there exists a permutation $\pi\colon [m] \to [m]$ such that $\tilde \bsig=\{\tilde \sigma_i\}_{i=1}^d$ defined as $\tilde \sigma_i = \{\pi(j)\colon j \in \sigma_i\}, i \in[d]$ is a \dimensionalpartition{} of the tensor $\tilde T$ defined as $\tilde T_{j_1,\ldots,j_m}= T_{j_{\pi(1)},\ldots,j_{\pi(m)}}$ for all $j_1,\ldots,j_m$. For instance if $T\in \R^{2\times 3 \times 2}$ and $\bsig = \{\{1,3\},\{2\}\}$, then one can define $\tilde T_{i,j,k}=T_{i,k,j}$ for all $i,j,k$ and $\tilde \bsig = \{\{1,2\},\{3\}\}$.

\Dimensionalpartition{}s are useful and convenient for describing all spectral systems of the same form as \eqref{eqsys} but for tensors of any order. To a given \dimensionalpartition{} $\bsig=\{\sigma_i\}_{i=1}^d$ of $T\in\R^{N_1 \times \ldots \times N_m}$ we associate the numbers 
%
$s_1,\ldots,s_{d}$, $\nu_1,\ldots,\nu_{d}$, and $n_1,\ldots,n_d$ defined as follows:
\begin{equation}\label{defs}
\nu_i=|\sigma_i|, \qquad s_i=\min\{a\mid a\in\sigma_i\}, \qquad  n_i = N_{s_i} \qquad \forall i \in[d],
\end{equation}
and $s_{d+1}=m+1$. We will always assume the definitions in \eqref{defs}, although the reference to the specific $\bsig$ will be understood implicitly. Moreover, for convenience, we will very often use the $n_i$ in place of the $N_i$. The relation between these two numbers is made more clear by noting that the dimensions $N_1\times \cdots \times N_m$ of $T$ can be rewritten as follows: 
\begin{align*}
\underbrace{\overbrace{\begin{matrix} 
N_1 \!\!\!\! & \times \!\!\!\! & \ldots \!\!\!\! & \times \!\!\!\! & N_{s_2-1}\\
\verteq & & & & \verteq \\
n_1  \!\!\!\! & \times \!\!\!\! & \ldots \!\!\!\! & \times  \!\!\!\!& n_1
\end{matrix}}^{\sigma_1}}_{\nu_1 \text{ times}} \,\,\,  \begin{matrix} \times \\ {\color{white} \verteq }\\ \times \end{matrix}\,\,\, \underbrace{\overbrace{\begin{matrix} 
N_{s_2} \!\!\!\! & \times \!\!\!\! & \ldots \!\!\!\! & \times \!\!\!\! & N_{s_3-1}\\
\verteq & & & & \verteq \\
n_2  \!\!\!\! & \times \!\!\!\! & \ldots \!\!\!\! & \times  \!\!\!\!& n_2
\end{matrix}}^{\sigma_2}}_{\nu_2 \text{ times}} \,\,\, \begin{matrix} \times\!\!\!\! & \ldots\!\!\!\! & \times \\
{\color{white} \verteq } & & {\color{white} \verteq }\\
\times\!\!\!\! & \ldots\!\!\!\! & \times \end{matrix}\,\,\, \underbrace{\overbrace{\begin{matrix} 
N_{s_d} \!\!\!\! & \times \!\!\!\! & \ldots \!\!\!\! & \times \!\!\!\! & N_{s_{d+1}-1}\\
\verteq & & & & \verteq \\
n_d  \!\!\!\! & \times \!\!\!\! & \ldots \!\!\!\! & \times  \!\!\!\!& n_d
\end{matrix}}^{\sigma_d}}_{\nu_d \text{ times}}
\end{align*}

Now, given $\bp=(p_1,\ldots,p_d)\in(1,\infty)^d$ and the \dimensionalpartition{} $\bsig$ of $T$, we define the Rayleigh quotient of $T$ induced by $\bsig$ and $\bp$ as follows:
\begin{equation} \label{genprojnorm}
\Phi(\bx_1,\dots,\bx_d)=\frac{f_T(\bx^{[\bsig]})}{\norm{\bx_1}_{p_1}^{\nu_1}\,\norm{\bx_2}_{p_2}^{\nu_2} \cdots\norm{\bx_d}_{p_d}^{\nu_d}}
\end{equation}
\begin{equation*}
\text{where}\qquad \bx^{[\bsig]}=(\overbrace{\bx_1,\ldots,\bx_1}^{\nu_1 \text{ times}},\overbrace{\bx_2,\ldots,\bx_2}^{\nu_2 \text{ times}},\ldots,\overbrace{\bx_d,\ldots,\bx_d}^{\nu_d \text{ times}})\, .
\end{equation*}
In particular,  note that the funtions $\Phi^1,\Phi^2,\Phi^3$ of \eqref{defRR} can be recovered by setting $\bsig^1 = \{\{1,2,3\}\}, \bp=p$, $\bsig^2 = \{\{1\},\{2,3\}\}, \bp=(p,q)$ and $\bsig^3 = \{\{1\},\{2\},\{3\}\}, \bp=(p,q,r)$, respectively. 

The Rayleigh quotient \eqref{genprojnorm} is naturally related to a norm of the tensor which depends on both the \dimensionalpartition{} $\bsig$ and the choice of the norms $\|\cdot\|_{p_i}$. We denote such norm as $\norm{T}_{(\bsig,\bp)}=\max_{\bx_1, \ldots, \bx_d\neq 0}|\Phi(\bx_1,\dots,\bx_d)|$. Note that the absolute value in the definition of $\norm{T}_{(\bsig,\bp)}$ can be omitted when $T$ is nonnegative. In fact, as discussed in the introduction, if $T$ is nonnegative, then the maximum is always attained at nonnegative vectors. 
In the case $d=m$ and $p_1=\ldots=p_m=2$, $\norm{T}_{(\bsig,\bp)}$ is called the spectral norm of $T$ and it is known that its computation is NP-hard in general (c.f. \cite{Lim2013}). If $d=m=2$, then $\norm{T}_{(\bsig,\bp)}$ coincides with the $\ell^{p,q}$-norm of the matrix $T$ \cite{Boyd} and it is also known to be NP-hard for general matrices if, for instance,  $p_1=p_2\neq 1,2$ is a rational number or $1\leq p_1<p_2\leq \infty$, see e.g. \cite{hendrickx2010matrix,Steinberg}.



A direct computation shows that the critical points of $\Phi$ in \eqref{genprojnorm} are solutions to the following spectral equation:
\begin{equation}\label{normspeceq}
\grad_{i} f_T(\bx^{[\bsig]}) = \lambda\psi_{p_i}(\bx_i), \qquad \norm{\bx_i}_{p_i}=1 \qquad \forall i \in [d],
\end{equation}
where $\grad_{i} f_T(\bx^{[\bsig]})\in\R^{n_i}$ denotes the gradient of the map $\bx_i\mapsto f_T(\bx^{[\bsig]})$, $\psi_{p_i}(\bx_i)=(|x_{i,1}|^{p_i-1}\sign(x_{i,1}),\ldots,|x_{i,n_i}|^{p_i-1}\sign(x_{i,n_i}))$ for all $\bx_i\in\R^{n_i}$ and $\sign(t)=t/|t|$ if $t\neq 0$ and $\sign(0)=0$. 

It is important to note that  $\grad_{i} f_T(\bx^{[\bsig]})$ and $\T_{s_i}(\bx^{[\bsig]})$ do not coincide in general,  unless $\nu_i=1$. 
Hence, we consider a more general class of spectral problems for tensors which is formulated as follows:
\begin{equation}\label{genspeceq}
\T_{s_i}(\bx^{[\bsig]}) = \lambda\psi_{p_i}(\bx_i), \qquad \norm{\bx_i}_{p_i}=1 \qquad \forall i \in [d].
\end{equation}



Depending on the choice of $\bsig$, various known spectral problems related to nonnegative tensors can be recovered from \eqref{genspeceq}. First, we note that if $m=2$ and $d=1$, then $T$ is matrix, $\bsig=\{\{1,2\}\}$ and with $p_1 = 2$ Equation \eqref{genspeceq} reduces to the standard eigenvector problem of $T$. If $m=2$ and $d=2$, then $T$ is matrix, $\bsig=\{\{1\},\{2\}\}$ and with $p_1=p_2=2$, \eqref{genspeceq} reduces to the standard singular vector problem of $T$. Furthermore, if $d=m$, then $\bsig=\{\{1\},\ldots,\{m\}\}$ and we recover equation (1.2) in \cite{Fried} which characterizes the $\ell^{p_1,\ldots,p_m}$-singular vectors of $T$. If $d=2$, then $\bsig=\{\{1,\ldots,k\},\{k+1,\ldots,m\}\}$ for some $k\in[m-1]$ and we recover equation (2) in \cite{Qi_rect} which characterizes the $\ell^{p_1,p_2}$-singular vectors of the rectangular tensor $T$. Finally, if $d=1$, then $\bsig=\{\{1,\ldots,m\}\}$ and we recover equation (7) in \cite{Lim} which characterizes the
$\ell^{p_1}$-eigenvectors of $T$. 
Perron-Frobenius type results have been established for each of the aforementioned spectral problems. In order to unify these results, we introduce here the following definition: 
\begin{defi}[$(\bsig, \bp)$-eigenvalues and eigenvectors]\label{defsigpeigen}
We say that $(\lambda,\bx)$ is a $(\bsig,\bp)$-eigenpair of $T$ if it satisfies \eqref{genspeceq}. We call $\lambda$ a $(\bsig,\bp)$-eigenvalue of $T$ and $\bx$ a $(\bsig,\bp)$-eigenvector of $T$. 
\end{defi}

Key assumptions in the Perron-Frobenius theory of nonnegative tensor are strict nonnegativity, weak irreducibility and (strong) irreducibility. 
In order to address the general spectral problem of Definition \ref{defsigpeigen}, we recast 	such assumptions in terms of the chosen \dimensionalpartition{}.

\begin{defi}[$\bsig$-nonnegativity and $\bsig$-irreducibility]\label{defirr}
For a nonnegative tensor $T\in\R^{N_1\times \ldots \times N_m}_+$ and an associated \dimensionalpartition{} $\bsig=\{\sigma_i\}_{i=1}^d$, consider the matrix $M\in\R^{(n_1+\ldots+n_d)\times (n_1+\ldots+n_d)}_+$ defined as 
$$
M_{(i,t_i),(k,l_k)}=\frac{\partial}{\partial x_{k,l_k}}\T_{s_i,t_i}(\bx^{[\bsig]})|_{\bx = \ones}\qquad \forall (i,t_i),(k,l_k)\in\I^{\bsig}=\bigcup_{i=1}^d \{i\}\times [n_i],
$$
where $\ones = (1,\ldots,1)^\top$ is the vector of all ones.
We say that $T$ is:
\begin{itemize}
\item $\bsig$-strictly nonnegative, if $M$ has at least one nonzero entry per row. 
\item  $\bsig$-weakly irreducible, if $M$ is irreducible.
\item  $\bsig$-strongly irreducible, if for every $\bx \in \R^{n_1}_+\times \ldots \times\R_+^{n_d}$ that is not entry-wise positive and is such that $\bx_i\neq 0$ for all $i\in [d]$,  
there exists $(k,l_k)\in\I^{\bsig}$ such that $x_{k,l_k}=0$ and $\T_{s_k,l_k}(\bx^{[\bsig]})>0$.
\end{itemize}
\end{defi}

These definitions coincide with most of the corresponding definitions introduced for the individual special cases. Indeed, if $d=1$, $\bsig$-strict nonnegativity reduces to the definition of strictly nonnegative tensor introduced in \cite{Hu2014}. If $d=1,2,m$, $\bsig$-weak irreducibility reduces to the definition of weak irreducibility introduced in \cite{Fried} and \cite{Qi_rect}, respectively. If $d=1,m$, $\bsig$-strong irreducibility reduces to the existing definitions of irreducibility introduced in \cite{Chang} and \cite{Fried}. However, in the case $d=2$, $\bsig$-strong irreducibility is strictly less restrictive than the definition of irreducibility introduced in \cite{Chang_rect_eig}. 
In Section \ref{irrsec} we give a detailed characterization of  each of these classes of nonnegative tensors. In particular, we propose equivalent formulations of these classes of tensors in terms of graphs and in terms of the entries of $T$. Furthermore, we show in Theorem \ref{irrsowirr} that $\bsig$-strong irreducibility implies $\bsig$-weak irreducibility which itself implies $\bsig$-strict nonnegativity. We also study how these classes are related, for a fixed tensor $T$ but different choices of $\bsig$.

Using different \dimensionalpartition{}s, one can associate several spectral problems  to a tensor $T$ via Definition \ref{defsigpeigen} and sometimes one can transfer properties that hold true for one formulation to another one. 
For instance, if a symmetric matrix $Q\in\R^{n\times n}_+$ is irreducible, i.e. $Q$ is $\{\{1,2\}\}$-irreducible, then its corresponding bipartite graph is strongly connected, i.e. $Q$ is also $\{\{1\},\{2\}\}$-irreducible. In particular, this implies that the classical Perron-Frobenius theorem holds not only for the eigenpairs of $Q$ but also for its singular pairs. A similar situation arises 
in the more general setting of tensors. In order to formalize this property,  we define the following partial order on the set of \dimensionalpartition{}s of $T$:
\begin{defi}\label{def:partial_order}
Let $\bsig=\{\sigma_i\}_{i=1}^d$, $\tilde \bsig=\{\tilde \sigma_i\}_{i=1}^{\tilde d}$ be two \dimensionalpartition{}s of $T\in\R^{N_1 \times \ldots \times N_m}$, then we write $ \bsig\sqsubseteq\tilde\bsig$ if $d\geq \tilde d$ and there exists $g\colon [d]\to [\tilde d]$ such that $\sigma_{i}\subset\tilde\sigma_{g(i)}$ for every $i\in[d]$. 
\end{defi}

Note, for example, that the \dimensionalpartition{}s $\bsig=\{\{1\},\{2\}\}$ and  $\tilde\bsig=\{\{1,2\}\}$ of the symmetric matrix $Q$ above, satisfy $\bsig\sqsubseteq\tilde \bsig$ and irreducibility with respect to $\bsig$ carries over to $\tilde \bsig$.
More generally, we discuss in Sections \ref{tnormsec} and \ref{sec:nonnegativeT} several properties of the tensor $T$ preserved by the partial ordering $\sqsubseteq$, that is properties that automatically hold for $\tilde \bsig$ when holding for a \dimensionalpartition{} $\bsig$ such that $\bsig \sqsubseteq \tilde \bsig$.  In particular, this is the case of tensor symmetries that we define below in terms of  $\bsig$. We first recall the  concept of partially symmetric tensors from \cite{SymFried1}:

\begin{defi}[Partially symmetric tensor, \cite{SymFried1}]\label{partsymtens}
Let $T\in\R^{N_1\times \ldots\times N_m}$ and let $\alpha\subset [m]$ be a subset of cardinality $2$ at least. We say that $T$ is symmetric with respect to $\alpha$ if $N_{i} = N_{i'}$ for each pair $\{{i},{i'}\}\subset \alpha$ and the value of $T_{j_1,\ldots,j_m}$ does not change if we interchange any two indices $j_{i},j_{i'}$ for $i,i'\in \alpha$ and any $j_k\in [N_{k}],k\in[m]$. By convention $T$ is symmetric with respect to each $\{i\}$ for $i\in[m]$. 
\end{defi}
\begin{defi}[$\bsig$-symmetry]
Let $T\in\R^{N_1\times \ldots\times N_m}$ and let $\bsig=\{\sigma_i\}_{i=1}^d$ be a \dimensionalpartition{} of $T$. We say that $T$ is $\bsig$-symmetric if it is partially symmetric with respect to $\sigma_i$ for all $i\in[d]$.
\end{defi}

Observe that, in particular, every matrix is $\{\{1\},\{2\}\}$-symmetric and symmetric matrices are 
$\{\{1,2\}\}$-symmetric. Moreover, if $T$ is $\bsig$-symmetric, then $T$ is $\tilde \bsig$-symmetric for every \dimensionalpartition{} $\tilde\bsig$ of $T$ such that $\tilde \bsig\sqsubseteq \bsig$.

\section{Main results}\label{mainresults}
In this section we describe the main results of this paper: A complete characterization of the irreducibility properties of $T$ in terms of the \dimensionalpartition{} $\bsig$; a unifying Perron-Frobenius theorem for the general tensor spectral problem of \eqref{genspeceq}; and  a generalized power method  with a linear convergence rates that allows to compute the dominant $(\bsig,\bp)$-eigenvalue and $(\bsig,\bp)$-eigenvector of $T$.
These results are based on a number of preliminary lemmas and results that we prove in the next sections. Thus, for the sake of readability, we postpone the poofs of the main results to the end of the paper. We devote this section to describe the results and to relate them with previous work. 

The first result is presented in the following: 
\begin{thm}\label{irrthm}
Let $T\in\R^{N_1\times \ldots\times N_m}_+$ and let $\bsig=\{\sigma_i\}_{i=1}^d$ and $\tilde \bsig=\{\tilde \sigma_i\}_{i=1}^{\tilde d}$ be \dimensionalpartition{}s of $T$ such that $\bsig \sqsubseteq\tilde \bsig$. Then, the following holds:
\begin{enumerate}[(i)]
\item If $T$ is $\bsig$-weakly irreducible, then $T$ is $\bsig$-strictly nonnegative.\label{irrthm0}
\item If $T$ is $\bsig$-strongly irreducible, then $T$ is $\bsig$-weakly irreducible.\label{irrthm1}
\item If $T$ is $\bsig$-strictly nonnegative, then $T$ is $\tilde\bsig$-strictly nonnegative.\label{irrthm2}
\item If $T$ is $\bsig$-weakly irreducible and $\tilde\bsig$-symmetric, then $T$ is $\tilde\bsig$-weakly irreducible.\label{irrthm3}
\item If $T$ is $\bsig$-strongly irreducible and $\tilde\bsig$-symmetric, then $T$ is $\tilde\bsig$-strongly irreducible.\label{irrthm4}
\end{enumerate}
\end{thm}
\begin{proof}
See Section \ref{irrsec}.
\end{proof}
Few comments regarding the partial symmetry assumption in \eqref{irrthm3} and \eqref{irrthm4} of the above theorem are in order: First, note that, as in the matrix case, the irreducibility of a tensor does not depend on the magnitude of its entries and so it is enough to assume that the nonzero pattern of $T$ is $\tilde \bsig$-symmetric. Second, by giving explicit examples, we note in Remarks \ref{sirrcex} and \ref{stirrcex} that the $\bsig$-symmetry assumption in \eqref{irrthm3} and \eqref{irrthm4}  can not be omitted in general. 

It is well known that in the case of nonnegative matrices, i.e. $m=2$ and $d=1$, $\bsig$-weak irreducibility and $\bsig$-strong irreducibility are equivalent. This equivalence is proved also for $m=2$ and $d=2$ in \cite[Lemma 3.1]{Fried}.
Furthermore, \eqref{irrthm0}, \eqref{irrthm1} are known for the particular cases $d=1$ and $d=m$. Precisely,  refer to \cite[Lemma 3.1]{Fried} for an equivalent of \eqref{irrthm1} and to  \cite[Proposition 8, (b)]{us} and  \cite[Corollary 2.1]{Hu2014} for an equivalent of \eqref{irrthm0} in the cases $d=1$ and $d=m$, respectively.
However, to our knowledge, the results of points \eqref{irrthm2}, \eqref{irrthm3}, \eqref{irrthm4} have not been proved before, in any setting.

\begin{figure}[t]
    \includegraphics[width=.3\textwidth,clip,trim=.0cm 0cm .0cm .0cm]{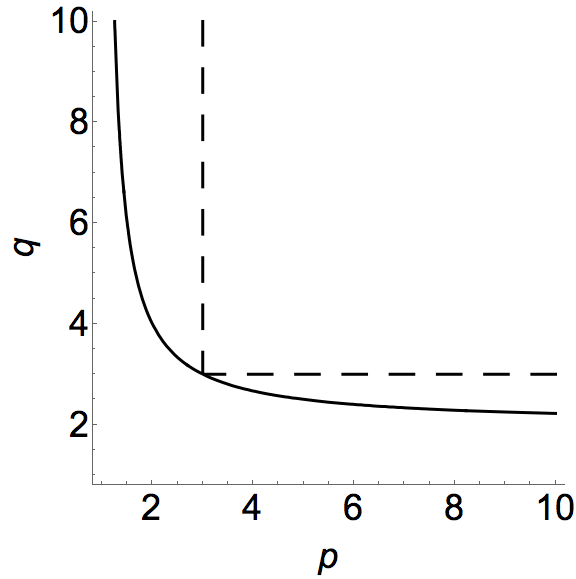}\hfill 
    \includegraphics[width=.26\textwidth,clip,trim=.0cm 0 .0cm 0]{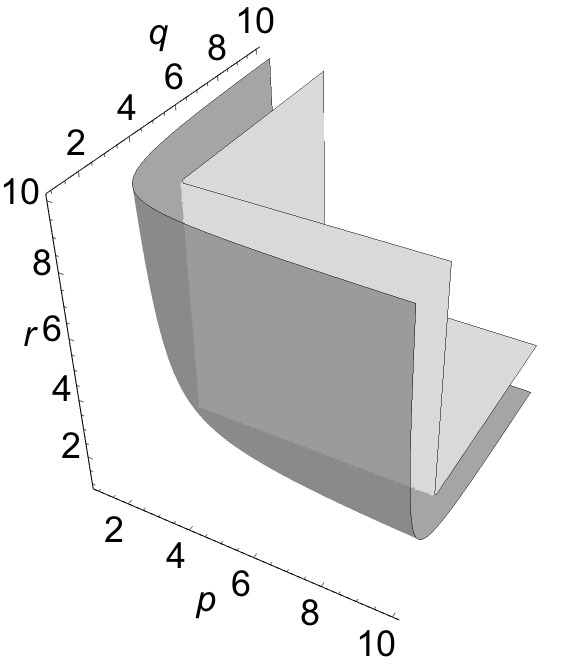}\hfill
    \includegraphics[width=.3\textwidth,clip,trim=.0cm 0cm .0cm .0cm]{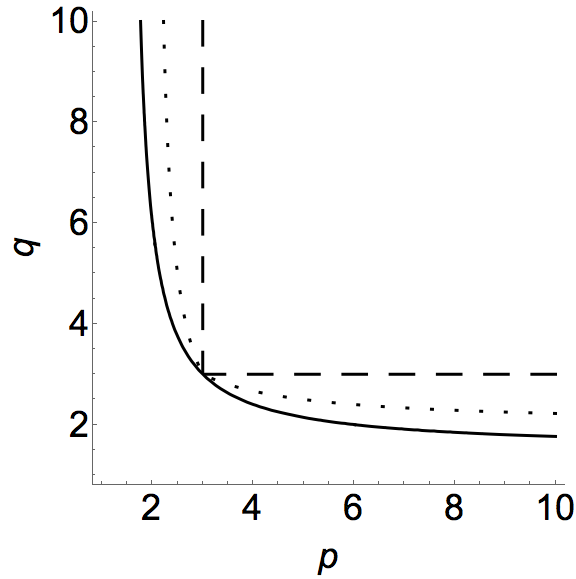}
    \caption{Conditions on $(p_1,p_2,p_3)=(p,q,r)$ for different settings involving a tensor of order $3$. The figure shows that generally, $\rho(A)\leq 1$ implies a less restrictive condition on $p,q,r$ than the previous existing ones. \textbf{Left:} Here $d=2$ so that $\bsig=\{\{1\},\{2,3\}\}$. The plain line is the set of $(p,q)$ such that $\rho(A)=1$ and the dashed line is the set of $(p,q)$ such that $\min\{p,q\}=3$ \cite{Qi_rect}. \textbf{Middle:} Here $d=3$ so that  $\bsig=\{\{1\},\{2\},\{3\}\}$. The dark gray surface is the set of $(p,q,r)$ such that $\rho(A)=1$ and the light gray surface is the set of $(p,q,r)$ such that $\min\{p,q,r\}=3$ \cite{Fried,Lim}. \textbf{Right:} Here $d=3$ again and $p$ is fixed to $p=3$. The plain line is the set of $(q,r)$ such that $\rho(A)=1$, the dotted line is the set of $(q,r)$ for which there exists $a\in\{p,q,r\}$ such that $2a\leq b(a-1)$ for all $b\in\{p,q,r\}\setminus\{a\}$ \cite{us} and the dashed line the set of $(q,r)$ such that $\min\{p,q,r\}=3$ \cite{Fried,Lim}.}
    \label{figurecond}
\end{figure}

Our second result is a new and unifying Perron-Frobenius theorem for $(\bsig,\bp)$-eigenpairs. First, let us consider the sets of nonnegative, nonnegative nonzero and positive tuples of vectors in $\R^{n_1}\times \ldots\times \R^{n_d}$, that is:
let $\kone_{+}^{\bsig}=\R^{n_1}_+\times \ldots \times\R_+^{n_d}$, $\kone^{\bsig}_{+,0}=\{\bx\in\kone^{\bsig}_+\mid \bx_i\neq 0, i\in[d]\}$ and let  $\kone^{\bsig}_{++}$ be the interior of $\kone_{+}^{\bsig}$. Furthermore,
let us define the $(\bsig,\bp)$-spectral radius of $T$: 
\begin{equation}\label{defspecrad}
r^{(\bsig,\bp)}(T) = \sup\big\{|\lambda| \,:\, \lambda \text{ is a }(\bsig,\bp)\text{-eigenvalue of }T\big\}.
\end{equation}
Note that if $m=2$, $d=1$ and $p_1=2$, then $r^{(\bsig,\bp)}(T)$ coincide with the spectral radius of the matrix $T$ and if $m=2$, $d=2$ and $p_1=p_2=2$, then $r^{(\bsig,\bp)}(T)$ coincide with the largest singular value of the matrix $T$.
As mentioned before, the key of our Perron-Frobenius theorem is the relation with the theory of multi-homogeneous and order-preserving mappings \cite{mhpfpaper}. In particular, let us consider
$F^{(\bsig,\bp)}\colon\kone_{+}^{\bsig}\to\kone_{+}^{\bsig}$ defined as $F^{(\bsig,\bp)}=(F_1^{(\bsig,\bp)},\ldots,F^{(\bsig,\bp)}_d)$ where $ F_i^{(\bsig,\bp)}=(F_{i,1}^{(\bsig,\bp)},\ldots,F_{i,n_i}^{(\bsig,\bp)})$ and, for all $(i,j_i)\in\I^{\bsig}$, 
\begin{equation}\label{defF}
F_{i,j_i}^{(\bsig,\bp)}(\bx) = \big(\T_{s_i,j_i}(\bx^{[\bsig]})\big)^{p_i'-1}, \qquad \text{with}\quad p_i'=\frac{p_i}{p_i-1}\, .
\end{equation}

We show in Lemma \ref{mhconnect} that the nonnegative $(\bsig,\bp)$-eigenpairs of $T$ are in bijection with the multi-homogeneous eigenvectors of $F^{(\bsig,\bp)}$, i.e. vectors $\bx\in\kone_{+,0}^{\bsig}$ for which there exists $\theta_1,\ldots,\theta_d\geq 0$ such that $F_i^{(\bsig,\bp)}(\bx)=\theta_i\bx_i$ for all $i\in[d]$. This key observation allows us to exploit the results proved in \cite{mhpfpaper}. In particular, we consider the homogeneity matrix $A(\bsig,\bp)\in\R_+^{d\times d}$ of $F^{(\bsig,\bp)}$ given as 
\begin{equation} \label{defA_first}
A(\bsig,\bp)= \diag(p_1'-1,\ldots,p_d'-1)(\ones\bnu^\top-I), \qquad \bnu=(|\sigma_1|,\ldots,|\sigma_d|)^\top,
\end{equation}
and let $\rho(A(\bsig,\bp))$ be its spectral radius. 
In the following, $A(\bsig,\bp)$ always refers to the homogeneity matrix of $F^{(\bsig,\bp)}$, hence, when it is clear from the context, we omit the arguments $(\bsig,\bp)$ and write $A$ instead of $A(\bsig,\bp)$. Note that the homogeneity matrix $A$ is always nonnegative and irreducible. Therefore, there exists a unique positive eigenvector $\bb$ such that $A^\top \bb =\rho(A)\bb$ with $\sum_{i=1}^d b_i =1$. Throughout the whole paper we will always devote the symbol $\bb$ to denote such a vector. 

Lemma 3.2 in \cite{mhpfpaper} implies that $\rho(A)$ is an upper bound on the Lipschitz constant of $F^{(\bsig,\bp)}$ with respect to a suitable weighted Hilbert  metric on $\kone^{\bsig}_{++}$.
Therefore, when $\bp$ and $\bsig$ are such that $\rho(A)\leq 1$, we can recast the $(\bsig,\bp)$-eigenvalue problem for $T$ in terms of the multi-homogeneous eigenvectors of a non-expansive map and derive the Perron-Frobenius theorem for $T$ as a consequence. 
In the particular cases $d=1$, $d=2$ and $d=m$, typical assumptions 
on $p_1,\ldots,p_m$ found in the literature on Perron-Frobenius theory of nonnegative tensors are $p_i\geq m$ for every $i\in[d]$, \cite{Fried,Lim,Qi_rect}. It is not difficult to see that if $p_i\geq m$ for all $i$, then 
$\rho(A)\leq 1$, with equality if and only if $p_1=\ldots=p_m=m$. 
However, by the Collatz-Wielandt formula, we have
$\rho(A)=\min_{\bv\in\R^d_{++}}\max_{i\in[d]}(A^\top\bv)_i/v_i$, 
and thus it is clear that there are many choices of $p_1,\ldots,p_d$ such that $\rho(A)\leq 1$ but $\min_{i\in[d]}p_i<m$.  Moreover, note that, as $A(\bsig,\bp)$ is irreducible, the function $(p_1,\ldots,p_d)\mapsto \rho(A(\bsig,\bp))$ is strictly monotonically decreasing in the sense that for every $\bp,\tilde \bp\in(1,\infty)^d$ with $\tilde p_i \leq p_i$ for all $i\in[d]$, it holds $\rho(A(\bsig,\tilde\bp))\geq\rho(A(\bsig,\bp))$ with equality if and only if $\bp=\tilde \bp$. An example comparing $\rho(A)\leq 1$ with the conditions on $p_1,\ldots,p_d$ given in \cite{Fried,us,Lim,Qi_rect} is shown in Figure \ref{figurecond}.

Our new Perron-Frobenius theorem consists of five parts: The first one is a weak Perron-Frobenius theorem ensuring the existence of a maximal nonnegative $(\bsig,\bp)$-eigenpair. The second characterizes $r^{(\bsig,\bp)}(T)$ via a Collatz-Wielandt formula, a Gelfand type formula and a cone spectral radius formula. The third part, 
gives sufficient conditions for the existence of a positive $(\bsig,\bp)$-eigenpair. The fourth part, gives conditions ensuring that $(\bsig,\bp)$-eigenvectors which are nonnegative but not positive can not correspond to $r^{(\bsig,\bp)}(T)$. The last part gives further conditions which guarantee that $T$ has a unique nonnegative $(\bsig,\bp)$-eigenvector.

Let us denote by $(F^{(\bsig,\bp)})^k$ the $k$-th composition of $F^{(\bsig,\bp)}$ with itself, that is $(F^{(\bsig,\bp)})^1(\bx)=F^{(\bsig,\bp)}(\bx)$ and $(F^{(\bsig,\bp)})^{k+1}(\bx)=F^{(\bsig,\bp)}((F^{(\bsig,\bp)})^{k}(\bx))$ for $k=1,2,\ldots$ 
Moreover, let us define the following product of balls 
$\S_+^{(\bp,\bsig)}=\big\{\bx\in\kone_{+}^{\bsig}\ \big|\ \norm{\bx_i}_{p_i}=1, \ \forall i \in[d]\big\}$ 
and its positive part $\S^{(\bp,\bsig)}_{++}=\S^{(\bp,\bsig)}_{+}\cap\kone_{++}^{\bsig}$.

\begin{thm}\label{PFtheory}
Let $\bsig=\{\sigma_i\}_{i=1}^d$ be a \dimensionalpartition{} of $T\in\R^{N_1\times \ldots\times N_m}_+$. Furthermore let $r^{(\bsig,\bp)}(T)$, $F^{(\bsig,\bp)}$, $A$ and $\bb$ be as in \eqref{defspecrad}, \eqref{defF} and \eqref{defA_first}, respectively. Suppose that $T$ is $\bsig$-strictly nonnegative and $\rho(A)\leq 1$. Then, the following properties hold: 
\begin{enumerate}[(i)]
\item There exists a $(\bsig,\bp)$-eigenpair $(\lambda,\bu)\in\R_+\times \kone_{+,0}^{\bsig}$ of $T$ such that $\lambda =r^{(\bsig,\bp)}(T)$. \label{PFweakexists}
\item
Let 
$\gamma = \frac{\sum_{i=1}^d b_ip_i'}{\sum_{i=1}^d b_ip_i'-1},$
then $\gamma \in (1,\infty)$ and the following Collatz-Wielandt formula holds:\label{PFweak}
\begin{align}\label{finalCW}
 \inf_{\bx\in\S_{++}^{(\bp,\bsig)}} \cwut(F^{(\bsig,\bp)},\bx) \, &=\, r^{(\bsig,\bp)}(T) \, =\, \max_{\by\in\S_+^{(\bp,\bsig)}}\cwlt(F^{(\bsig,\bp)},\by) \\
\text{where} \qquad
\cwut(F^{(\bsig,\bp)},\bx)&=\prod_{i=1}^d\Big(\max_{j_i\in[n_i]}\frac{F^{(\bsig,\bp)}_{i,j_i}(\bx)}{x_{i,j_i}}\Big)^{(\gamma-1)b_i}\notag\\ 
 \text{and}\qquad \cwlt(F^{(\bsig,\bp)},\by) &=\prod_{i=1}^d\Big(\min_{\substack{j_i\in[n_i], \ y_{i,j_i}>0}}\frac{F_{i,j_i}^{(\bsig,\bp)}(\by)}{y_{i,j_i}}\Big)^{(\gamma-1)b_i}.\notag 
\end{align}
If additionally, $\rho(A)=1$, then it holds
\begin{align}\label{Gelf}
r^{(\bsig,\bp)}(T)&= \sup_{\bz\in\kone_{+,0}^{\bsig}}\limsup_{k\to\infty}\Big(\prod_{i=1}^d\norm{(F^{(\bsig,\bp)})^k_i(\bz)}_{p_i}^{b_i}\Big)^{\frac{\gamma-1}{k}}\\ &=\lim_{k\to\infty} \Big( \sup_{\bz\in\S^{(\bsig,\bp)}_{+}}\prod_{i=1}^d\norm{(F^{(\bsig,\bp)})^k_i(\bz)}_{p_i}^{b_i}\Big)^{\frac{\gamma-1}{k}}.\notag
\end{align}
\item If either $\rho(A)<1$ or $T$ is $\bsig$-weakly irreducible, then the $(\bsig,\bp)$-eigenvector $\bu$ of \eqref{PFweakexists} can be chosen to be strictly positive, i.e. $\bu\in\kone_{++}^{\bsig}$. Moreover, $\bu$ is then the unique positive $(\bsig,\bp)$-eigenvector of $T$.\label{uniquepos}
\item If $T$ is $\bsig$-weakly irreducible, then for every $(\bsig,\bp)$-eigenpair $(\vartheta,\bx)$ of $T$ such that $\bx\in \kone_{+,0}\setminus\kone_{++}$, it holds $\vartheta <  r^{(\bsig,\bp)}(T)$.\label{stricteval}
\item If $T$ is $\bsig$-strongly irreducible, then the $(\bsig,\bp)$-eigenvector $\bu$ of \eqref{PFweakexists} is positive and it is the unique nonnegative $(\bsig,\bp)$-eigenvector of $T$.\label{irrres}
\end{enumerate}
\end{thm}
\begin{proof}
See Section \ref{proofsec}.
\end{proof}
Note that Theorem \ref{PFtheory}, \eqref{PFweakexists} is relatively obvious when $T$ is $\bsig$-symmetric. In fact, as shown in Lemma \ref{invsym}, in this case $r^{(\bsig,\bp)}(T)=\norm{T}_{(\bsig,\bp)}$ and thus the existence of $\bu$ follows from the fact that a continuous function over a compact domain attains its maximum. In particular, this is always the case when $m=d$. However, when $d\neq m$ and $T$ is not $\bsig$-symmetric, proving the existence of $\bu$ is more delicate. The cases $d=1,2$ are proved in  \cite[Theorem 2.3]{Yang1} and \cite[Theorem 4.2]{Qi_rect}, but under the assumption that 
$p_i\geq m$, for all $i \in[d]$. 
Our Theorem \ref{PFtheory}, instead, addresses a more general case, but requires $\bsig$-strict nonnegativity of $T$. Although this is an additional requirement, we show for instance in Example \ref{weaksneg} that this is a very mild assumption.

A particularly interesting consequence of the Collatz-Wielandt formula \eqref{finalCW} is that every positive $(\bsig,\bp)$-eigenvector of $T$ must correspond to the maximal eigenvalue $r^{(\bsig,\bp)}(T)$. Such formula is proved in  \cite[Theorem 2.3]{Yang1} and  \cite[Theorem 1]{us} for the cases $d=1$ and $d=m$, respectively. Both assume that $T$ has a positive $(\bsig,\bp)$-eigenvector and either $p_1\geq m$ if $d=1$, or $(m-1)p_j \leq p_k(p_j-1)$ for some $j\in[d]$ and all $k\in[m]\setminus\{j\}$, if $d=m$. 
In the case $d=2$, a similar formula is proved in \cite[Theorem 4.6]{Yang2} under the assumption that $T$ is $\bsig$-strongly irreducible and $p_1=p_2=m$. It is not difficult to see that that all the above conditions on $p_1,\ldots,p_d$ imply $\rho(A)<1$ except when $p_1=\ldots=p_d=m$, in which case $\rho(A)=1$ (see Figure \ref{figurecond}).
Hence, the assumption in Theorem \ref{PFtheory} is generally less restrictive than any known counterpart.

To our knowledge, \eqref{stricteval} of Theorem \ref{PFtheory} and the characterizations of the spectral radius in \eqref{Gelf} have not been proved before, besides the  particular cases $d=1$ and $p_1=m$. In fact, the only result comparable with point \eqref{stricteval} we are aware of is Theorem 2.4 in \cite{Yang1}, where it is proved that if all the entries of $T$ are strictly positive, $d=1$ and $p_1=m$, then $r^{(\bsig,\bp)}(T)$ is geometrically simple, i.e. for every $(\bsig,\bp)$-eigenvalue $\lambda$ of $T$ with $\lambda \neq r^{(\bsig,\bp)}(T)$ it holds $|\lambda| <r^{(\bsig,\bp)}(T)$. As for the characterization in \eqref{Gelf}, we are only aware of a brief discussion involving the Gelfand formula in \cite[Section 2]{NLA:NLA1902}.

 Finally, \eqref{irrres} of  Theorem \ref{PFtheory}  is a well known result for the cases $d=1$, $d=2$ and $d=m$, see e.g. \cite[Theorem 1.4]{Chang} and \cite[Theorem 14]{us}. Indeed, this result follows from the fact that every nonnegative $(\bsig,\bp)$-eigenvector of $T$ has positive entries and its proof holds regardless of the choice of $p_1,\ldots,p_d\in(1,\infty)$.

Our last main contribution concerns the computational aspects of the positive $(\bsig,\bp)$-eigenvector $\bu$ in Theorem \ref{PFtheory}. This vector can be computed using a nonlinear generalization of the power method. The classical power method allows to compute the leading eigenvector of a primitive matrix $M$ via the iterative sequence $\bx^{k+1}=\frac{M\bx^k}{\norm{M\bx^k}_2}$ for any positive starting point $\bx^0$.
The power method for general $(\bsig,\bp)$ tensor eigenpairs is formulated as follows: Let $\bx^{0}\in\kone_{++}^{\bsig}$ and, for $k=0,1,2,\ldots$, define
\begin{equation}\label{normalPM}
\bx^{k+1}=\left(\frac{F_1^{(\bsig,\bp)}(\bx^k)}{\norm{F_1^{(\bsig,\bp)}(\bx^k)}_{p_1}},\ldots,\frac{F_d^{(\bsig,\bp)}(\bx^k)}{\norm{F_d^{(\bsig,\bp)}(\bx^k)}_{p_d}}\right)\, .
\end{equation}
This sequence provides a natural generalization of the power method for computing eigenpairs of matrices and it reduces to the one proposed in \cite{NQZ}, \cite{Chang_rect_eig}, \cite{Fried} for the cases $d=1$, $d=2$ and $d=m$, respectively.  Usually, convergence towards $\bu$ is only guaranteed when $\rho(A)\leq 1$ and the Jacobian matrix of $F^{(\bsig,\bp)}$ is primitive. However, we prove that when $\rho(A)< 1$ it is sufficient that $T$ is $\bsig$-strictly nonnegative, or equivalently the matrix $M$ of Definition \ref{defirr} has at least one positive entry per row, so that the sequence converges towards $\bu$ with a linear convergence rate. 

If $\rho(A)=1$, primitivity can be relaxed into irreducibility by considering a different sequence, which we define in the following. Let $G^{(\bsig,\bp)}\colon\kone_{+}^{\bsig}\to \kone_{+}^{\bsig}$ be defined as $G^{(\bsig,\bp)}=(G^{(\bsig,\bp)}_1,\ldots,G^{(\bsig,\bp)}_d)$, $G^{(\bsig,\bp)}_i=(G^{(\bsig,\bp)}_{i,1},\ldots,G^{(\bsig,\bp)}_{i,n_i})$ and 
\begin{equation}\label{defG} 
G^{(\bsig,\bp)}_{i,j_i}(\bx)=\sqrt{x_{i,j_i}F^{(\bsig,\bp)}_{i,j_i}(\bx)} \qquad \forall (i,j_i)\in\I^{\bsig},
\end{equation}
and consider the sequence 
\begin{equation}\label{newPM}
\by^{k+1}=\left(\frac{G_1^{(\bsig,\bp)}(\by^k)}{\norm{G_1^{(\bsig,\bp)}(\by^k)}_{p_1}},\ldots,\frac{G_d^{(\bsig,\bp)}(\by^k)}{\norm{G_d^{(\bsig,\bp)}(\by^k)}_{p_d}}\right),
\end{equation}
where $k=0,1,2,\ldots$ and $\bz^{0}\in\kone_{++}^{\bsig}.$

The convergence of the two sequences in \eqref{normalPM}  and \eqref{newPM} is proved in the next Theorem \ref{PMthm}. In order to facilitate its statement, for $k\geq 1$, we let
\begin{equation*}
\begin{array}{ll}\widehat\xi_k=\cwlt(F^{(\bsig,\bp)},\bx^k),&\qquad \widecheck\xi_k=\cwut(F^{(\bsig,\bp)},\bx^k), \\ \widehat\zeta_k=\big(\cwut(G^{(\bsig,\bp)},\by^k)\big)^2,&\qquad \widecheck\zeta_k=\big(\cwut(G^{(\bsig,\bp)},\by^k)\big)^2,\end{array}
\end{equation*}
where $\cwlt,\cwut$ are defined as in Theorem \ref{PFtheory} and recall from \cite{mhpfpaper} the definition of the weighted Hilbert metric: 
$$\mu_{\bb}(\bx,\by)=\sum_{i=1}^db_i\ln\Big(\max_{j_i,l_i\in[n_i]}\frac{x_{i,j_i}y_{i,l_i}}{y_{i,j_i}x_{i,l_i}}\Big)\qquad \forall x,y\in\kone_{++}^{\bsig}.$$
Note that, by the continuity of $\cwlt,\cwut$ in $\kone_{++}^{\bsig}$, if the sequence of $\bx^k$, resp.\ $\by^k$,  
converges to a positive $(\bsig,\bp)$-eigenvector $\bu$ of $T$, then $\displaystyle\lim_{k\to\infty}\widehat\xi_k=\lim_{k\to\infty} \widecheck\xi_k=r^{(\bsig,\bp)}(T)$, resp.\ $\displaystyle\lim_{k\to\infty}\widehat\zeta_k=\lim_{k\to\infty} \widecheck\zeta_k=r^{(\bsig,\bp)}(T)$. 
\begin{thm}\label{PMthm}
Assume that $T$ is $\bsig$-strictly nonnegative and has a positive $(\bsig,\bp)$-eigenvector $\bu$ and $\rho(A)\leq 1$. Furthermore, let $(\bx^k)_{k=0}^\infty$, $(\by^k)_{k=0}^\infty$, $(\widehat\xi_k)_{k=1}^\infty$ ,$(\widecheck\xi_k)_{k=1}^\infty$, $(\widehat\zeta_k)_{k=1}^\infty$ and $(\widecheck\zeta_k)_{k=1}^\infty$ be as above. Then, the following holds:
\begin{enumerate}[(i)]
\item If $\omega\in\{\xi,\zeta\}$, then for all $k=1,2,\ldots$ it holds\label{genPM}
\begin{equation}\label{monoseq}
\widehat\omega_k\, \leq \, \widehat\omega_{k+1} \, \leq \, r^{(\bsig,\bp)}(T) \, \leq \, \widecheck\omega_{k+1} \, \leq \, \widecheck\omega_{k}, 
\end{equation}
and for every $\epsilon >0$, if $\widehat\omega_{k} -\widecheck\omega_{k} <\epsilon$, then
\begin{equation}\label{stopcrit}
\Big|\frac{\widehat\omega_{k} +\widecheck\omega_{k}}{2}-r^{(\bsig,\bp)}(T)\Big| \leq \frac{\epsilon}{2}.
\end{equation} 
\item If $\rho(A)<1$, then $\lim_{k\to\infty} \bx^k = \bu$ and, with $\bb$ as in Theorem \ref{PFtheory},\label{FPM} 
\begin{equation}\label{convrate}
\mu_{\bb}(\bx^k,\bu) \leq \bigg(\frac{\mu_{\bb}(\bx^1,\bx^0)}{1-\rho(A)}\bigg) \rho(A)^k\qquad \forall k=1,2,\ldots
\end{equation}
\item If $T$ is $\bsig$-weakly irreducible, then $\lim_{k\to\infty}\by^k = \bu$.\label{GPM}
\end{enumerate}
\end{thm}
\begin{proof}
See Section \ref{proofsec}.
\end{proof}
To our knowledge, the convergence of the power method for nonnegative tensors has been analyzed only for the cases $d=1$, $d=2$ and $d=m$.

If $d=1$, the known assumptions for the convergence of the power method 
towards $\bu$ are either $p_1> m$ and $M$ primitive (\cite[Corollary 5.1]{Fried}), where $M$ is as in Definition \ref{defirr}, or $p_1=m$ and $M$ irreducible (see \cite[Theorem 5.4]{Hu2014}). Clearly, if
$p_1>m$, then the assumptions of Theorem \ref{PMthm}, \eqref{FPM} are considerably weaker as we only assume $T$ to be $\bsig$-strictly nonnegative. When $p_1=m$, Theorem \ref{PMthm}, \eqref{GPM} is equivalent to \cite[Theorem 5.4]{Hu2014} in terms of assumptions. However,  note that the method in \cite{Hu2014} uses an additive shift while we have a multiplicative shift. Furthermore,  the convergence rate of \cite{Fried} for the case $p_1\geq m$ holds only asymptotically and assumes $T$ to be $\bsig$-weakly irreducible. Whereas, a linear convergence rate for the case $p_1=m$ is proved under the assumption that $M$ is primitive in  \cite[Theorem 4.1]{Hu2014}.

For $d=2$, results are known only in the case $p_1=p_2=m$. Precisely, in Theorem 7 of \cite{Chang_rect_eig} it is proved that $(\bx^k)_{k=1}^{\infty}$ converges towards $\bu$ if $p_1=p_2=m$ and $T$ is irreducible in the sense of Definition 1 in \cite{Chang_rect_eig} which, as discussed above, is more restrictive than $T$ being $\bsig$-strongly irreducible. As $\bsig$-strong irreducibility implies $\bsig$-weak irreducibility, it is clear that Theorem \ref{PMthm}, \eqref{GPM} improves these results. A linear convergence rate is proved in \cite[Theorem 4]{linlks} for the case where $p_1=p_2=m$ but requires additional assumptions on $T$.  

Finally, if $d=m$, then it is proved in  \cite[Theorem 2]{us} that a variation of the power method converges to $\bu$ under the condition that $T$ is $\bsig$-weakly irreducible and $(m-1)p_j \leq p_k(p_j-1)$ for some $j\in[d]$ and all $k\in[m]\setminus\{j\}$, which, as discussed above, implies $\rho(A)<1$ unless $p_1=\ldots=p_d=m$, in which case $\rho(A)=1$. Hence, in terms of convergence assumptions, Theorem \ref{PMthm} improves \cite[Theorem 2]{us}. However, when $p_1=\ldots=p_d=m$, the latter result provides an asymptotic convergence rate which is not implied by Theorem \ref{PMthm}.


\section{Tensor norms and spectral problems}\label{tnormsec}
In this section we study a number of relations between the critical points of the Rayleigh quotient $\Phi$ in \eqref{genprojnorm} and the $(\bsig,\bp)$-eigenpairs of $T$. The goal of this discussion is twofold. First, it gives an optimization perspective on $(\bsig,\bp)$-eigenpairs and second it explains how to use our main results, in particular Theorem \ref{PMthm}, for the computation of $\norm{T}_{(\bsig,\bp)}$. Recall that, here and throughout the manuscript, we denote by $\grad_{i} f_T(\bx^{[\bsig]})\in\R^{n_i}$ the gradient of the map $\bx_i\mapsto f_T(\bx^{[\bsig]})$.

In a first step, we prove in Lemma \ref{sym} how to construct a $\bsig$-symmetric tensor $S\in\R^{N_1\times \ldots \times N_m}$ so that $f_T(\bx^{[\bsig]})=f_S(\bx^{[\bsig]})$ and $\grad_i f_T(\bx^{[\bsig]})=\nu_i\mathcal{S}_{s_i}(\bx^{[\bsig]})$ for every $\bx\in\R^{n_1}\times \ldots \times \R^{n_d}$, where $\mathcal{S}(\bz)=\grad f_S(\bz)$ for every $\bz\in\R^{N_1}\times \ldots \times \R^{N_m}$. This construction has practical relevance, as it allows for a simple implementation of $\grad_i f_T(\bx^{[\bsig]})$ and it shows that partial symmetry is relevant when computing the critical points of $\Phi$. Furthermore, as $f_T=f_S$, we note that $S$ can be used in place of $T$ in the definition of $\Phi$, without changing the optimization problem. In particular, we have $\norm{T}_{(\bsig,\bp)}=\norm{S}_{(\bsig,\bp)}$.

In a second step, we prove in Lemma \ref{invsym} that the $(\bsig,\bp)$-eigenvector and $(\bsig,\bp)$-eigenvalues of the $\bsig$-symmetric tensor $S$ are precisely the critical points, resp. values, of $\Phi$. In particular, this means that $\norm{S}_{(\bsig,\bp)}=r^{(\bsig,\bp)}(S)$ and thus, if $S$ satisfies the assumptions of Theorem \ref{PMthm},  the power method converges to a global maximizer $\bu$ of $\Phi$ and $f_T(\bu^{[\bsig]})=\norm{T}_{(\bsig,\bp)}$.

Finally, in Lemma \ref{banachlemma} we discuss cases where  $\norm{T}_{(\bsig,\bp)}=\norm{T}_{(\tilde\bsig,\tilde\bp)}$ for different \dimensionalpartition{}s $\bsig$, $\tilde \bsig$.

\begin{lem}\label{sym}
Let $\bsig=\{\sigma_i\}_{i=1}^d$ be a \dimensionalpartition{} of $T\in\R^{N_1\times \ldots \times N_m}$. For $i\in[d]$, let $\mathfrak{S}_i$ be the permutation group of $\sigma_i$, and define $S\in\R^{N_1\times \ldots \times N_m}$ as
\begin{equation}\label{defsym}
S_{j_1,\ldots,j_m}=\sum_{i=1}^d\frac{1}{\nu_i !}\sum_{\pi_i\in\mathfrak{S}_i}T_{j_{\pi_1(s_1)},\ldots,j_{\pi_1(s_2-1)},\ldots,j_{\pi_d(s_{d})},\ldots,j_{\pi_d(s_{d+1}-1)}}
\end{equation}
for all $j_k\in[N_k]$, $k\in[m]$. Then, we have $f_T(\bx^{[\bsig]}) =f_S(\bx^{[\bsig]})$ for all $\bx$. Furthermore, $S$ is $\bsig$-symmetric and, with $\mathcal S=\grad f_S$, it holds
$\grad_i f_T(\bx^{[\bsig]})=\nu_i\mathcal{S}_{s_i}(\bx^{[\bsig]})$ for all $\bx$ and $i \in[d]$. 
\end{lem}
\begin{proof}
For $\bx\in\R^{n_1}\times\ldots\times\R^{n_d}$, let $Z\in\R^{N_1\times \ldots \times N_m}$ be the tensor defined as $Z_{j_1,\ldots,j_m}=T_{j_1,\ldots,j_m}\prod_{i\in[d]}\prod_{t\in\sigma_i}x_{i,j_t}$ for all $j_1,\ldots,j_m$. We have
\begin{align*}
f_S(\bx^{[\bsig]}) &=
\sum_{i\in[d]}\sum_{t\in\sigma_i}\sum_{j_{t}\in [n_i]}Z_{j_1,\ldots,j_m}\\ 
&= \sum_{i\in[d]}\sum_{t\in\sigma_i}\sum_{j_{t}\in [n_i]}\sum_{a\in[d]}\frac{1}{|\mathfrak{S}_a|}\sum_{\pi_a\in \mathfrak{S}_a}Z_{j_1,\ldots, j_{s_a-1},j_{\pi_a(s_a)},\ldots,,j_{\pi_a(s_{a+1}-1)},j_{s_{a+1}},\ldots,j_m}\\
&= \sum_{i\in[d]}\sum_{t\in\sigma_i}\sum_{j_{t}\in [n_i]}S_{j_1,\ldots,j_m}x_{1,j_1}\cdots x_{d,j_m}=f_S(\bx^{[\bsig]}).
\end{align*}
To conclude, note that, as $S$ is partially symmetric with respect to $\sigma_i$, Equation (4) in \cite{Lim} implies $ \nu_i\mathcal{S}_i(\bx^{[\bsig]})=\grad_if_S(\bx^{[\bsig]})=\grad_if_T(\bx^{[\bsig]})$.
\end{proof}
Now, we show that the converse of Lemma \ref{sym} is also true.
\begin{lem}\label{invsym} Let $\bsig=\{\sigma_i\}_{i=1}^d$ be a \dimensionalpartition{} of $T\in\R^{N_1\times\ldots\times N_m}$ and $\bp\in(1,\infty)^d$. If $T$ is $\bsig$-symmetric, then the $(\bsig,\bp)$-eigenvectors of $T$ are critical points of the Rayleigh quotient $\Phi$ defined in \eqref{genprojnorm}. Furthermore, it holds $\norm{T}_{(\bsig,\bp)}=r^{(\bsig,\bp)}(T)$.
\end{lem}
\begin{proof}
As $T$ is symmetric with respect to $\sigma_i$ for $i\in [d]$, we have $S=T$ where $S$ is as in \eqref{defsym}. Thus Lemma \ref{sym} implies that $\grad_i f_T(\bx^{[\bsig]})=\nu_i\T_{s_i}(\bx^{[\bsig]})$ for every $i$. Hence, if $\bx\in\kone_{+}^{\bsig}$ satisfies 
$\T_{s_i}(\bx^{[\bsig]})=\lambda \psi_{p_i}(\bx_i)$ for all $i$, then we have $\grad_i f_T(\bx^{[\bsig]})=\nu_i\lambda \psi_{p_i}(\bx_i)$ for every $i$, i.e. $\bx$ is a critical point of $\Phi$. Finally, note that $\lambda$ is the critical value associated to $\bx$ since $f_T(\bx^{[\bsig]})=\langle\T_{s_i}(\bx^{[\bsig]}),\bx_i\rangle=\lambda \norm{\bx_i}_{p_i}^{p_i}=\lambda,$
 as $f_T(\bz)$ is linear in $\bz_{s_i}$. Therefore, we have $\norm{T}_{(\bsig,\bp)}=r^{(\bsig,\bp)}(T)$.
\end{proof}

Finally, we show below that if $\bsig=\{\sigma_i\}_{i=1}^d$, $\tilde\bsig=\{\tilde\sigma_i\}_{i=1}^{\tilde d}$ are \dimensionalpartition{}s of $T$, $\bsig \sqsubseteq \tilde\bsig$ and $T$ is partially symmetric with respect to $\tilde\bsig$, then the corresponding tensor norms coincide for suitable choices of the $p_i,\tilde p_i$. This result is essentially a corollary of Theorem 1 in \cite{BanachA}.
\begin{lem}\label{banachlemma}
Let $\bsig=\{\sigma_i\}_{i=1}^d$, $\tilde\bsig=\{\tilde \sigma_i\}_{i=1}^{\tilde d}$ be two \dimensionalpartition{}s of $T\in\R^{N_1\times \ldots \times N_m}$. If $\bsig \sqsubseteq \tilde\bsig$ and $\bp\in(1,\infty)^d$, $\tilde\bp\in(1,\infty)^{\tilde d}$ are such that $p_i=\tilde p_j$ whenever $\sigma_i\subset\tilde\sigma_j.$ Then we have
$\norm{T}_{(\bsig,\bp)}=\norm{T}_{(\tilde\bsig,\tilde \bp)}$.
\end{lem}
\begin{proof}
If $\bsig=\tilde \bsig$, there is nothing to prove, so let us assume $\bsig\neq\tilde \bsig$. Clearly, we have $\norm{T}_{(\bsig,\bp)}\geq \norm{T}_{(\tilde \bsig,\tilde\bp)}$. We prove the reverse inequality. First, note that by Lemma \ref{sym}, by substituting $T$ with $S$ if necessary, we may assume without loss of generality that $T$ is $\tilde\bsig$-symmetric. Now, let $(\bx_1^*,\ldots,\bx_d^*)$ be such that $\norm{\bx_i^*}_{p_i}=1$ for all $i\in[d]$ and $\norm{T}_{(\bsig,\bp)}=f_T((\bx^*)^{[\bsig]})$. As $\bsig \sqsubseteq \tilde\bsig$ and $\bsig\neq\tilde \bsig$, there exists $i,j\in [d],k\in [\tilde d]$ such that $i<j$, $\sigma_i\subset \tilde \sigma_k$ and $\sigma_j\subset \tilde \sigma_k$. Then $p_i=p_j$ by assumption and we have 
\begin{equation*}
\norm{T}_{(\bsig,\bp)}=\max_{\bx_i,\bx_j\neq 0} \frac{f_T\big((\bx_1^*,\ldots,\bx_{i-1}^*,\bx_i,\bx^*_{i+1},\ldots,\bx_{j-1}^*,\bx_j,\bx^*_{j+1},\ldots,\bx^*_d)^{[\bsig]}\Big)}{\norm{\bx_i}_{p_i}^{\nu_i}\,\norm{\bx_j}_{p_j}^{\nu_j}}
\end{equation*}
where $\nu_i=|\sigma_i|$ and $\nu_j=|\sigma_j|$. Now, as $T$ is partially symmetric with respect to $\tilde \sigma_k$, Theorem 1 in \cite{BanachA} implies that we there exists $(\by_1^*,\ldots,\by_d^*)$ with $\norm{\by_l^*}_{p_l}=1, l\in[d]$ such that $\norm{T}_{(\bsig,\bp)}=f_T((\by^*)^{[\bsig]})$ and $\by_i^* = \by_j^*$. Continuing this argument for every $i,j\in [d],k\in [\tilde d]$ as above, we deduce that there exists $(\bz_1^*,\ldots,\bz_d^*)$ with $\norm{\bz_l^*}_{p_l}=1, l\in[d]$, $\norm{T}_{(\bsig,\bp)}=f_T((\bz^*)^{[\bsig]})$ and the following property: For every $i,j\in[d]$ such that there exists $k\in[\tilde d]$ with $\sigma_i\subset \tilde \sigma_k$ and $\sigma_j\subset \tilde \sigma_k$, it holds $\bz_i^* = \bz_j^*$. It follows that there exists $\bar\bsig$ and $\bar\bz\in\kone_{+,0}^{\tilde\bsig}$ such that $\bar\bz^{[\bar\bsig]}=\bz^*$. Hence, we have $$\norm{T}_{(\bsig,\bp)}=f_T((\bz^*)^{[\bsig]})=f_T((\bar\bz^{[\bar \bsig]})^{[\bsig]})=f_T(\bar\bz^{[\tilde \bsig]})\leq \norm{T}_{(\tilde \bsig,\tilde\bp)},$$
which concludes the proof.
\end{proof}
%
%
%
%
%
%
%
%
%
%
\section{The multi-homogeneous setting}\label{mhsection}
One of the keys of our Perron--Frobenius theorem is the use of the multi--homogeneous map $F^{(\bsig,\bp)}$, defined in \eqref{defF}. In the matrix case $M\in\R^{n\times n}$ we know that the eigenvectors of $M$ are fixed points of the homogeneous map $\bx\mapsto M\bx$  in the projective space of $\R^n$, that is, if $M\bx=\lambda \bx$ for some $\bx$ with $\norm{\bx}=1$, then $g(\bx)=\bx$ where $g(\bz)=M\bz/\norm{M\bz}$. We extend this observation to the tensor setting by means of $F^{(\bp,\bsig)}$. Precisely, we prove in Lemma \ref{mhconnect} that the $(\bsig,\bp)$-eigenvectors of $T$ are exactly the fixed points of $F^{(\bsig,\bp)}$ in the product of projective spaces corresponding to
$\R^{n_1}\times \ldots \times \R^{n_d}$. This observation is useful as, for nonnegative tensors $T$, the mapping $F^{(\bsig,\bp)}$ is  order-preserving and multi-homogeneous and thus we can apply the nonlinear Perron-Frobenius theorem discussed in \cite{mhpfpaper} to derive conditions on the dominant $(\bsig,\bp)$-eigenpair of nonnegative tensors.  In particular, the spectral radius of $F^{(\bsig,\bp)}$ is strictly related to the $(\bsig,\bp)$-spectral radius of $T$ and the irreducibility conditions of $F^{(\bsig,\bp)}$ transfer to $T$. 

Let us first review two important properties of $F^{(\bsig,\bp)}$, together with some useful related notation borrowed from \cite{mhpfpaper}. The proof of these properties follows by a straightforward computation and is omitted for brevity. 

Assume that $T$ is a nonnegative tensor, $\bsig$ a shape partition of $T$ and $\bp =(p_1,\dots,p_d)\geq 1$. Then:

1. $F^{(\bsig,\bp)}$ is order--preserving, that is 
$$
\bx\lek \by \qquad \implies \qquad F^{(\bsig,\bp)}(\bx)\lek F^{(\bsig,\bp)}(\by) \qquad \forall\, \bx,\by\in\kone^{\bsig}_+\, ,
$$
where, for $\bx,\by\in\kone^{\bsig}_+$, we write $\bx\lek\by$ if $\by-\bx\in\kone^{\bsig}_{+}$. This  partial ordering notation is particularly useful and throughout we also write  $\bx\lekk\by$ and $\bx\lekkk\by$ if $\by-\bx\in\kone^{\bsig}_{+}\setminus\{0\}$ and $\by-\bx\in\kone^{\bsig}_{++}$, respectively. 

2. $F^{(\bsig,\bp)}$ is multi-homogeneous with homogeneity matrix $A=A(\bsig,\bp)$, where $A(\bsig,\bp)$ is defined in \eqref{defA_first}. 
In particular, for any $\bx \in \kone^{\bsig}_{+}$ and any $\bt \in \R^d_+$ it holds 
$$
F^{(\bsig,\bp)}(\bt \otimes \bx) = \bt^A \otimes F^{(\bsig,\bp)}(\bx)\, ,
$$
where $\bt \otimes \bx = (\theta_1\bx_1, \dots, \theta_d \bx_d)$ and $\bt^A$ is the vector with entries $(\bt^A)_i = \prod_{j=1}^d \theta_j^{A_{ij}}$. 

A vector $\bx$ is an eigenvector of $F^{(\bsig,\bp)}$ with (vector--valued) eigenvalue $\bt \in \R^d$ if $F^{(\bsig,\bp)}(\bx) = \bt \otimes \bx$. The following lemma establishes the correspondence between the nonnegative eigenvectors and eigenvalues of $F^{(\bsig,\bp)}$  and the nonnegative $(\bsig,\bp)$-eigenvectors and $(\bsig,\bp)$-eigenvalues of $T$. 
 \begin{lem}\label{mhconnect}
 Let $\bx\in\kone_{+,0}^{\bsig}$, then the following two statements are equivalent: 
 \begin{enumerate}[$\qquad(a)$]
\item $\big(\frac{\bx_1}{\norm{\bx_1}_{p_1}},\ldots,\frac{\bx_d}{\norm{\bx_d}_{p_d}}\big)$ is a $(\bsig,\bp)$-eigenvector of $T$.\label{equivsig1}
\item There exists $\bt\in\R^d_+$ such that $F^{(\bsig,\bp)}(\bx) = \bt \otimes \bx$. \label{equivsig2} 
\end{enumerate}
Furthermore, suppose that $\norm{\bx_i}_{p_i}=1$, then we have the following:
\begin{enumerate}[$\qquad(a)$]
 \setcounter{enumi}{2}
\item If $F^{(\bsig,\bp)}(\bx) = \bt \otimes \bx$, there exists $\lambda\in\R_+$ such that $\theta_{i}=\lambda^{p_i'-1}$ for all $i\in[d]$ and $(\lambda,\bx)$ is a $(\bsig,\bp)$-eigenpair of $T$.\label{evalmhtotens}
\item If $\lambda \geq 0$ is such that $(\lambda,\bx)$ is a $(\bsig,\bp)$-eigenpair of $T$, then $F^{(\bsig,\bp)}(\bx) = \tilde{\boldsymbol \lambda} \otimes \bx$ with $\tilde{\lambda}_i = \lambda^{p_i'-1}$, for all $i\in [d]$. 
\label{evaltenstomh}
\end{enumerate}
 \end{lem}
 \begin{proof}
 Let $\nu_i=|\sigma_i|$ for all $i\in[d]$.
 First assume that $\tilde \bx=\big(\frac{\bx_1}{\norm{\bx_1}_{p_1}},\ldots,\frac{\bx_d}{\norm{\bx_d}_{p_d}}\big)$ is a $(\bsig,\bp)$-eigenvector of $T$, then there exists $\lambda \geq 0$ such that for every $i\in[d]$, it holds
 \begin{equation*}
\lambda\norm{\bx_i}_{p_i}^{1-p_i}\psi_{p_i}(\bx_i)=\lambda \psi_{p_i}(\tilde \bx_i) = \T_{s_i}(\tilde\bx^{[\bsig]})= \Big(\norm{\bx_i}_{p_i}\prod_{j=1}^d\norm{\bx_j}^{-\nu_j}_{p_j}\Big) \T_{s_i}(\bx^{[\bsig]}).
 \end{equation*}
By rearranging the above equation and composing it by $\psi_{p_i'}$, we get
\begin{equation*}\label{tensvaltomhval}
F^{(\bsig,\bp)}_i(\bx) =\psi_{p_i'}\big(\T_{s_i}(\bx^{[\bsig]})\big)= \Big(\lambda\norm{\bx_i}_{p_i}^{-p_i}\prod_{j=1}^d\norm{\bx_j}^{\nu_j}_{p_j}\Big)^{p_i'-1} \bx_i = \theta_i \, \bx_i
\end{equation*}
and thus \eqref{equivsig1} implies \eqref{equivsig2}. In particular, note that if $\norm{\bx_i}_{p_i}=1$ for all $i\in[d]$, then \eqref{evaltenstomh} follows from the above equation.

Now suppose that there exists $\bt\in\R^d_+$ such that $F^{(\bsig,\bp)}(\bx) = \bt \otimes \bx$ 
and set $\tilde \bx=\Big(\frac{\bx_1}{\norm{\bx_1}_{p_1}},\ldots,\frac{\bx_d}{\norm{\bx_d}_{p_d}}\Big)$. Then, we have $F^{(\bsig,\bp)}(\tilde \bx) = \tilde \bt \otimes \tilde \bx$
where $\tilde \theta$ is defined as  $\tilde \theta_i = \theta_i\norm{\bx_i}_{p_i}^{p_i'}\big(\prod_{j=1}^d \norm{\bx_j}_{p_j}^{-\nu_j}\big)^{p_i'-1}$  for all $i\in[d]$. Hence, we get 
\begin{equation*}\T_{s_i}(\tilde \bx^{[\bsig]})=\psi_{p_i}\big(F^{(\bsig,\bp)}_i(\tilde\bx) \big)= \tilde\theta_i^{p_i-1}\psi_{p_i}(\tilde\bx_i) \qquad \forall i \in[d].
\end{equation*}
To conclude, we prove that there exists $\lambda\geq 0$ such that $\tilde\theta_i^{p_i-1}=\lambda $ for all $i\in[d]$. This follows from the fact that $\T_{s_i}(\tilde\bx^{[\bsig]})=\tilde\theta_i^{p_i-1}\psi_{p_i}(\tilde\bx_i)$ as it implies that
\begin{equation*}
f_T(\tilde \bx^{[\bsig]})=\ps{\tilde \bx_i}{\T_{s_i}(\tilde\bx^{[\bsig]})} = \tilde\theta_i^{p_i-1} \ps{\tilde \bx_i}{\psi_{p_i}(\tilde\bx_i)}=\tilde\theta_i^{p_i-1}\norm{\tilde \bx_i}_{p_i}^{p_i}=\tilde\theta_i^{p_i-1}.
\end{equation*}
Finally, if $\norm{\bx_i}_{p_i}=1$ for all $i$ and $\lambda =f_T(\tilde \bx^{[\bsig]})$, we have $\tilde \theta_i =\theta_i=\lambda^{p_i'-1}$ for all $i \in [d]$, which proves \eqref{evalmhtotens}.
\end{proof}

We can now show the connection between the spectral radius of the order preserving  multi-homogeneous mapping $F^{(\bsig,\bp)}$ and the $(\bsig,\bp)$-spectral radius of the tensor $T$. To this end, let us denote by $\S_{+}^{(\bsig,\bp)}$ the product of $p_i$-spheres in $\kone_+^{\bsig}$, i.e. $\S_{+}^{(\bsig,\bp)}=\{\bx\in\kone_{+}^{\bsig}\mid \norm{\bx_i}_{p_i}=1, \ i\in[d]\}$.
The spectral radius of $F^{(\bsig,\bp)}$ is defined as (see \cite[Section\ 4]{mhpfpaper}), 
\begin{equation*}
r_{\bb}(F^{(\bsig,\bp)})=\sup\Big\{ \prod_{i=1}^d \theta_i^{b_i}\ \Big|\ F^{(\bsig,\bp)}(\bx)=\bt\krog\bx \text{ for some } \bx\in\S_+^{(\bsig,\bp)}\Big\},
\end{equation*}
where we recall that $\bb\in\R^d$ is the unique positive eigenvector of $A^\top$ such that $\sum_{i=1}^d b_i=1$.
We relate $r_{\bb}(F^{(\bsig,\bp)}) $ and $r^{(\bsig,\bp)}(T)$ in the following: 
\begin{lem}\label{specradlem}
Let $(\lambda,\bx)\in\R_+\times \S_+^{(\bsig,\bp)}$ be a  $(\bsig,\bp)$-eigenpair of $T$ such that $\lambda = r^{(\bsig,\bp)}(T)$, and let  $(\bt,\by)\in\R^d_+\times \S_+^{(\bsig,\bp)}$ be such that $F^{(\bsig,\bp)}(\by)=\bt \krog \by$ with  $\prod_{i=1}^d \theta_i^{b_i}=r_{\bb}(F^{(\bsig,\bp)})$. Then
\begin{equation}\label{realtespecrad}r^{(\bsig,\bp)}(T)=r_{\bb}(F^{(\bsig,\bp)})^{\gamma-1}, \quad \text{where}\quad \gamma = \frac{\sum_{i=1}^db_ip_i'}{{\sum_{i=1}^db_ip_i'}-1}\, .
\end{equation}
\end{lem}
\begin{proof}
Let $\gamma' = \sum_{i=1}^db_ip_i'$, then $\gamma'\geq \min_{j\in[d]}p_j'\sum_{i=1}^d b_i= \min_{j\in[d]}p_j'>1$, thus $\gamma = \frac{\gamma'}{\gamma'-1}\in(1,\infty)$ and $(\gamma-1)(\gamma'-1)=1$. Now, Lemma \ref{mhconnect}, \eqref{evaltenstomh} implies that $F^{(\bsig,\bp)}(\bx)=\blam\krog \bx$ with $\lambda_i=\lambda^{p_i'-1}$, hence we have
\begin{equation}\label{leqspecrad}
r^{(\bsig,\bp)}(T)^{\gamma'-1}=\lambda^{\gamma'-1}=\prod_{i=1}^d\lambda^{b_i(p_i'-1)} =\prod_{i=1}^d\lambda_i^{b_i}\leq r_{\bb}(F^{(\bsig,\bp)}).
\end{equation}
On the other hand, by Lemma \ref{mhconnect}, \eqref{evalmhtotens} we know that there exists $\theta\in\R_+$ such that $\theta_i=\theta^{p_i'-1}$ for all $i\in[d]$ and $\theta$ is a $(\bsig,\bp)$-eigenvalue of $T$. Hence, we have
\begin{equation}\label{geqspecrad}
r_{\bb}(F^{(\bsig,\bp)}) =\prod_{i=1}^d \theta_i^{b_i}=\prod_{i=1}^d \theta^{b_i(p_i'-1)}=\theta^{\gamma'-1}\leq r^{(\bsig,\bp)}(T)^{\gamma'-1}.
\end{equation}
\end{proof}

 \section{Classes of nonnegative tensors}\label{sec:nonnegativeT}
 We discuss here the different classes of nonnegative tensors given in Definition \ref{defirr}. We propose characterizations in terms of graphs for each of them and explain how they relate to a number of structural properties of the corresponding multi--homogeneous mapping $F^{(\bsig,\bp)}$. To this end, we first introduce the $\bsig$-graph of a nonnegative tensor $T$ and discuss some of its properties. Then, we analyze each of the nonnegative tensors classes in a separate subsection,  we show how they are relate with each other and we conclude with the proof of Theorem \ref{irrthm}. 
 
 \subsection{$\bsig$-graphs of nonnegative tensors}
 We propose a definition of graph associated to a nonnegative tensor and with respect to one of its shape partitions. We call this graph the $\bsig$-graph of $T$ and denote it $\G^{\bsig}(T)$. Simply put, the set of nodes of $\G^{\bsig}(T)$ is $\I^{\bsig}$ and there is an edge from $(k,l_k)$ to $(i,t_i)$, if the variable $x_{i,j_i}$ effectively appear in the expression of $\T_{s_k,l_k}(\bx^{[\bsig]})$. Formally, we have the following:
 \begin{defi}[$\bsig$-graph of a nonnegative tensor] Let $\bsig=\{\sigma_i\}_{i=1}^d$ be a \dimensionalpartition{} of $T\in\R^{N_1\times \ldots \times N_m}_+$. The $\bsig$-graph of $T$ is the directed graph $\G^{\bsig}(T)=(\I^{\bsig},\E^{\bsig}(T))$ defined as follows: The set of nodes is $\I^{\bsig}= \cup_{i=1}^d \{i\}\times [n_i]$ and there is an edge $\big((k,l_k),(i,t_i)\big)\in\E^{\bsig}(T)\subset \I^{\bsig}\times \I^{\bsig}$ if one of the following condition holds:
 \begin{itemize} 
 \item $(k,l_k)\neq (i,t_i)$ and there exists $j_1,\ldots,j_m$ such that $T_{j_1,\ldots,j_m}>0$, $j_{s_k}=l_k$ and $t_i\in\{j_a\mid a\in\sigma_i\}$. 
 \item $(k,l_k)= (i,t_i)$ and there exists $j_1,\ldots,j_m$ such that $T_{j_1,\ldots,j_m}>0$, $j_{s_k}=l_k$ and $t_i\in\{j_a\mid a\in\sigma_i\setminus\{s_i\}\}$.
 \end{itemize}
 \end{defi}
 Note that in the cases $d=1$ and $d=m$, $\G^{\bsig}(T)$ coincides with the graphs associated to $T$ introduced in Sections 4 and 1 of \cite{Fried}, respectively. Furthermore, when $d=2$, $\G^{\bsig}(T)$ coincides with the graph associated to $T$ introduced in Section 4 of \cite{Qi_rect}.  
 In particular, if $M\in\R^{n\times n}$ is a square matrix, then the \dimensionalpartition{}s of $M$ are $\bsig=\{\{1,2\}\}$ and $\tilde \bsig=\{\{1\},\{2\}\}$, and  $\G^{\bsig}(M)$ is the graph with $n$ nodes and adjacency matrix $M$, whereas $\G^{\tilde \bsig}(M)$ is the bipartite graph with $2n$ nodes and adjacency matrix $\begin{bsmallmatrix} 0 & M^\top \\ M & 0\end{bsmallmatrix}$. In the next example we illustrate the three graphs associated with a square tensor of order $3$.
\definecolor{myblue}{RGB}{255,255,255}
\definecolor{myred}{RGB}{255,255,255} 
\definecolor{mygray}{RGB}{255,255,255}
 \begin{ex}\label{graphex}
Let $T\in\R^{3\times 3 \times 3}$ be defined as
\begin{equation*}
    T_{2,2,1}=T_{3,2,1}=T_{1,3,1}=T_{2,2,2}=T_{1,1,3}=1\quad \text{and}\quad T_{i,j,k}=0 \quad \text{otherwise.}
\end{equation*}
Furthermore, let $\bsig^1,\bsig^2,\bsig^3$ be the \dimensionalpartition{}s of $T$, namely:
\begin{equation}\label{shapex}
\bsig^1 = \{\{1,2,3\}\},\qquad \bsig^2 = \{\{1\},\{2,3\}\} \qquad\text{and}\qquad\bsig^3=\{\{1\},\{2\},\{3\}\}.
\end{equation}
The following three $\bsig$-graphs can be associated to $T$:
\begin{center}
\begin{minipage}{.3\textwidth}
\begin{center}
\begin{tikzpicture}[scale =0.2, ->,>=stealth']
\tikzset{
    state/.style={
    circle,
           draw=black, thick,
           minimum height=1em,
           inner sep=2pt,
           text centered
           },
}
\node[state,anchor=center,fill=mygray] (1box) 
   {1};
\node[state, right of = 1box,fill=mygray] (2box) 
  {2};
\node[state, right of = 2box,fill=mygray] (3box) 
  {3};
  \node[state, right of = 3box, xshift=-0.1cm,draw=none] 
    {$\scriptstyle\{1,2,3\}$};
\draw [line width=.2mm] (1box) to [out=240,in=300,looseness=6, preaction={-triangle 90,thin,draw,shorten >=-1mm}] (1box);
\draw [line width=.2mm] (2box) to [out=240,in=300,looseness=6, preaction={-triangle 90,thin,draw,shorten >=-1mm}] (2box);
  \path[draw=black,solid,line width=.2mm,fill=black, preaction={-triangle 90,thin,draw,shorten >=-1mm}]
 (2box) edge (1box) 
 (3box) edge (2box) 
 ;
 \path[<->,draw=black,solid,line width=.2mm,fill=black, preaction={-triangle 90,thin,draw,shorten >=-1mm}]
 (1box) edge[bend left=60] (3box);
\end{tikzpicture}

$\G^{\bsig^1}(T)$ $\qquad$
\end{center}
\end{minipage}
\hfill
\begin{minipage}{.29\textwidth}
\begin{center}
\begin{tikzpicture}[scale =0.2, ->,>=stealth']
\tikzset{
    state/.style={
    circle,
           draw=black, thick,
           minimum height=1em,
           inner sep=2pt,
           text centered
           },
}
\node[state,  draw=none, anchor = center] (cent) {};
\node[state, above of = cent, yshift=.06mm, fill=mygray] (12box) 
   {2};
\node[state, left of = 12box,fill=mygray] (11box) 
  {1};
\node[state, right of = 12box,fill=mygray] (13box) 
  {3};
   \node[state, right of = 13box,xshift=-0.1cm,draw=none] 
    {$\scriptstyle\{1\}\phantom{,2,3}$};
  \node[state, below of = 12box, yshift=-.12mm,fill=myblue] (22box) 
     {2};
  \node[state, left of = 22box,fill=myblue] (21box) 
    {1};
  \node[state, right of = 22box,fill=myblue] (23box) 
    {3};
 \node[state, right of = 23box,draw=none,xshift=-0.1cm] 
    {$\scriptstyle\{2,3\}\phantom{,1}$};
\draw [line width=.22mm] (22box) to [out=240,in=300,looseness=6, preaction={-triangle 90,thin,draw,shorten >=-1mm}] (22box);
\path[<->,draw=black,solid,line width=.2mm,fill=black, preaction={-triangle 90,thin,draw,shorten >=-1mm}]
 (21box) edge (11box)  
 (23box) edge (11box)    
 (12box) edge[bend left=20] (22box)   
 (13box) edge (22box)   
 (21box) edge[bend right=60] (23box) 
 ;
 \path[draw=black,solid,line width=.2mm,fill=black, preaction={-triangle 90,thin,draw,shorten >=-1mm}]
 (12box) edge[bend right=5] (21box)   
  (13box) edge (21box)    
  (22box) edge (21box)   
  ;
\end{tikzpicture}

$\G^{\bsig^2}(T)$ $\qquad$
\end{center}
\end{minipage}
\hfill
\begin{minipage}{.3\textwidth}
\begin{center}
\begin{tikzpicture}[scale =0.2, ->,>=stealth']
\tikzset{
    state/.style={
    circle,
           draw=black, thick,
           minimum height=1em,
           inner sep=2pt,
           text centered
           },
}
\node[state,anchor=center, xshift = -0.7cm, yshift = 1.14cm,fill=mygray] (11box) 
   {1};
\node[state,anchor=center, xshift = -1.02cm, yshift = 0.59cm,fill=mygray] (12box) 
   {2};
\node[state,anchor=center, xshift = -1.33cm, yshift = 0.04cm,fill=mygray] (13box) 
   {3};
\node[state,anchor=center, xshift = -0.63cm, yshift = -1cm,fill=myblue] (21box) 
   {1};
\node[state,anchor=center, xshift = 0cm, yshift = -1cm,fill=myblue] (22box) 
   {2};
\node[state,anchor=center, xshift = 0.63cm, yshift = -1cm,fill=myblue] (23box) 
   {3};
  \node[state,anchor=center, xshift = 1.33cm, yshift = 0.04cm,fill=myred] (31box) 
      {1};
   \node[state,anchor=center, xshift = 1.02cm, yshift = 0.59cm,fill=myred] (32box) 
      {2};
   \node[state,anchor=center, xshift = 0.7cm, yshift = 1.14cm,fill=myred] (33box) 
      {3};
 \node[state, anchor = center,draw=none, xshift=-1.65cm, yshift = -0.51cm] 
    {$\scriptstyle\{1\}$};     
    \node[state, anchor = center,draw=none, xshift=1.27cm, yshift = -1.02cm]
    {$\scriptstyle\{2\}$};   
    \node[state, anchor = center,draw=none, xshift=0.38cm, yshift = 1.69cm]
    {$\scriptstyle\{3\}$};   
\path[<->,draw=black,solid,line width=.2mm,fill=black, preaction={-triangle 90,thin,draw,shorten >=-1mm}]
(12box) edge[bend left = 10] (22box) 
(22box) edge[bend left = 3] (31box) 
(12box) edge[bend right = 12.5] (31box)
(13box) edge[bend left = 3.5] (22box) 
(13box) edge[bend right = 16.3] (31box)
(11box) edge[bend left = 0] (23box) 
(23box) edge[bend right = 3] (31box) 
(11box) edge[bend right = 8] (31box)
(22box) edge[bend left = 7] (32box) 
(12box) edge (32box)
(11box) edge[bend left = 10] (21box) 
(21box) edge (33box) 
(11box) edge (33box)
;   
\end{tikzpicture}

$\quad$ $\G^{\bsig^3}(T)$
\end{center}
\end{minipage}

\end{center}
 \end{ex}
%
%
%
%


In the following lemma we show that the $\bsig$--graph of a nonnegative tensor $T$  coincides with the graph of the corresponding multi--homogeneous mapping $F^{(\bsig,\bp)}$. Moreover, we prove that the Jacobian of the map is always an adjacency matrix for such a graph. To this end, we first recall from \cite{mhpfpaper} the definition of graph of a multi--homogeneous map. 
\begin{defi}[Graph of a multi-homogeneous mapping]The graph $\G(F)$ of an order-preserving multi-homogeneous mapping $F\colon\kone_{+}^{\bsig}\to\kone_{+}^{\bsig}$ is the pair $\G(F)=(\I^{\bsig},\E(F))$, where $\I^{\bsig}$ is the set of nodes and an edge  $\big((k,l_k),(i,j_i)\big)\in\E(F)$ exists if and only if $\lim_{z\to\infty}F_{k,l_k}(\be^{(i,j_i)}(z))=\infty$, where $\be^{(k,l_k)}\colon\R_+\to\kone_{+}^{\bsig}$ is defined as $(\be^{(i,j_i)}(z))_{i,j_i}=z$ and $(\be^{(i,j_i)}(z))_{\eta,t_\eta}=1$ for all $(\eta,t_\eta)\in\I^{\bsig}\setminus\{(i,j_i)\}$.
\end{defi}
\begin{lem}\label{characgraph}
	Let $\G^{\bsig}(T)=(\I^{\bsig},\E^{\bsig}(T))$ be the $\bsig$-graph of $T$ and $\G(F^{(\bsig,\bp)})=(\I^{\bsig},\E(F^{(\bsig,\bp)}))$ be the graph of $F^{(\bsig,\bp)}$ as multi-homogeneous mapping. Then, for every $(k,l_k),(i,t_i)\in\I^{\bsig}$, the following are equivalent:
	\begin{enumerate}[(i)]
		\item $\big((k,l_k),(i,t_i)\big)\in\E^{\bsig}(T)$.\label{edge}
		\item For all $\bx\in\kone_{++}^{\bsig}$, $\tfrac{\partial}{\partial x_{i,t_i}}F_{k,l_k}^{(\bsig,\bp)}(\bx)$ exists and it holds $\tfrac{\partial}{\partial x_{i,t_i}}F_{k,l_k}^{(\bsig,\bp)}(\bx)>0$.\label{jacoentry}
		\item $\big((k,l_k),(i,t_i)\big)\in\E(F^{(\bsig,\bp)})$.\label{mhedge}
	\end{enumerate}
\end{lem}
\begin{proof}
	\eqref{edge}$\Rightarrow$\eqref{jacoentry}: If $\big((k,l_k),(i,t_i)\big)\in\E^{\bsig}(T)$, there exist indexes $j_1,\ldots,j_m$ such that $T_{j_1,\ldots,j_m}>0$, $j_{s_k}=l_k$ and $t_i\in \varsigma_i$ where $\varsigma_i=\{j_{a}\mid a\in\sigma_i\}$ if $(k,l_k)\neq(i,t_i)$ and $\varsigma_i=\{j_{a}\mid a\in\sigma_i\sauf\{s_i\}\}$ if $(k,l_k)=(i,t_i)$. Then, for $\bx\in\kone^{\bsig}_{++}$, we have
		\begin{equation*} 
		F_{k,l_k}^{(\bsig,\bp)}(\bx)= \big(\T_{s_k,l_k}(\bx^{[\bsig]})\big)^{p_k'-1}\geq \Big(T_{j_1,\ldots,j_m}\xi^{(\bsig)}_{j_1,\ldots,j_m}(\bx)\Big)^{p_k'-1}>0,
		\end{equation*}
		$$\text{with}\qquad\xi^{(\bsig)}_{j_1,\ldots,j_m}(\bx)=\Big(\prod_{l=1, l \neq k}^d\prod_{t\in\sigma_l}x_{l,j_t}\Big) \prod_{t\in\sigma_k, t\neq s_k}x_{k,j_t}.$$
		It follows that $\tfrac{\partial}{\partial x_{i,t_i}}F_{k,l_k}^{(\bsig,\bp)}(\bx)$ exists since $\bx \mapsto \T_{s_k,l_k}(\bx^{[\bsig]})$ is a polynomial and for all $\alpha >0$, $z\mapsto z^{\alpha}$ is differentiable at $z>0$. It holds
		\begin{align}\label{gradfsig}
		\frac{\partial}{\partial x_{i,t_i}}F_{k,l_k}^{(\bsig,\bp)}(\bx)&=(p_k'-1)\big(\T_{s_k,l_k}(\bx^{[\bsig]})\big)^{p_k'-2}\,\frac{\partial}{\partial x_{i,t_i}}\T_{s_k,l_k}(\bx^{[\bsig]})\\
		&\geq (p_k'-1)\big(\T_{s_k,l_k}(\bx^{[\bsig]})\big)^{p_k'-2} T_{j_1,\ldots,j_m}\frac{\partial}{\partial x_{i,t_i}}\xi^{(\bsig)}_{j_1,\ldots,j_m}(\bx).\notag
		\end{align}
		It holds $\frac{\partial}{\partial x_{i,t_i}}\xi^{(\bsig)}_{j_1,\ldots,j_m}(\bx)>0$ since $t_i\in \varsigma_i$ and $x\in\kone^{\bsig}_{++}$. Hence, $\frac{\partial}{\partial x_{i,t_i}}F_{k,l_k}^{(\bsig,\bp)}(\bx)>0$.
		\eqref{jacoentry}$\Rightarrow$\eqref{mhedge}: Let $z>0$, then $\be^{(i,t_i)}(z)\in\kone_{++}^{\bsig}$ so that $\frac{\partial}{\partial x_{i,t_i}}F_{k,l_k}^{(\bsig,\bp)}(\be^{(i,t_i)}(z))>0$. The equality in \eqref{gradfsig}, implies that $$0<\frac{\partial}{\partial x_{i,t_i}}\T_{s_k,l_k}\big((\be^{(i,t_i)}(z))^{[\bsig]}\big)=\frac{d}{d\,z}f(z)\quad\text{with}\quad f(z)=\T_{s_k,l_k}\big((\be^{(i,t_i)}(z))^{[\bsig]}\big).$$ 
		It follows that $f$ is a nonconstant polynomial in $z$ with nonnegative coefficients and thus $\lim_{z\to\infty}f(z)=\infty$. As $F_{k,l_k}^{(\bsig,\bp)}(\be^{(i,t_i)}(z))=f(z)^{p_k'-1}$ and $p_k'-1>0$ we have $\big((k,l_k),(i,t_i)\big)\in\E(F^{(\bsig,\bp)})$.
	\newline
	\eqref{mhedge}$\Rightarrow$\eqref{edge}: We prove that if \eqref{edge} does not hold, then \eqref{mhedge} does not hold either. Indeed, if $\big((k,l_k),(i,t_i)\big)\notin\E^{\bsig}(T)$, then by construction of $\be^{(i,t_i)}(z)$, with $j_{s_k}=l_k$ we have
	$F_{k,l_k}^{(\bsig,\bp)}(\be^{(i,t_i)}(z)) = \big(\sum_{l=1, l\neq s_k}^m\sum_{j_l=1}^{N_l}T_{j_1,\ldots,j_m}\big)^{p_k'-1}.$
	As this expression is bounded and constant in $z$, \eqref{mhedge} can not hold.
\end{proof}
 
 \subsection{$\bsig$-strict nonnegativity}\label{strictnnegsec}
The $\bsig$-strict nonnegativity condition for a nonnegative tensor corresponds to the requirement that the associated multi-homogeneous map $F^{(\bsig,\bp)}$ is positive, i.e.\ $F^{(\bsig,\bp)}(\bx) \in \kone^{\bsig}_{++}$ for every $\bx\in\kone^{\bsig}_{++}$ This is shown by the following 
 \begin{lem}\label{characpos}
 The followings are equivalent:
 \begin{enumerate}[(i)]
 \item $T$ is $\bsig$-strictly nonnegative. \label{pos0}
 \item $F^{(\bsig,\bp)}(\bx)\in\kone^{\bsig}_{++}$ for every $\bx\in\kone^{\bsig}_{++}$,\label{pos1}
 \item $F^{(\bsig,\bp)}(\ones)\in\kone^{\bsig}_{++}$,\label{pos2}
 \item For every $(i,l_i)\in \I^{\bsig}$, there exists $j_1,\ldots,j_m$ with $T_{j_1,\ldots,j_m}>0$ and $j_{s_i}=l_i$.\label{pos3}
 \end{enumerate}
 \end{lem}
 \begin{proof}
 $\eqref{pos0}\Rightarrow\eqref{pos1}$: Let $(i,l_i)\in\I^{\bsig}$ and $\bx\in\kone_{++}^{\bsig}$, we show that $F^{(\bsig,\bp)}_{i,l_i}>0$. As $T$ is $\bsig$-strictly nonnegative, there exists $(k,j_k)\in\I^{\bsig}$ such that the matrix $M$ of Definition \ref{defirr} satisfies $M_{(i,l_i),{(k,j_k)}}>0$. Lemma \ref{characgraph} then implies $\frac{\partial}{\partial x_{k,j_k}}\T_{s_i,l_i}(\bx^{[\bsig]})>0$ and so
 \begin{equation*}
F^{(\bsig,\bp)}_{i,l_i}(\bx)=\big(\T_{s_i,l_i}(\bx^{[\bsig]})\big)^{p_i'-1}=\Big(\frac{1}{\nu_k}\sum_{l_k=1}^{n_k}\frac{\partial}{\partial x_{k,l_k}}\T_{s_i,l_i}(\bx^{[\bsig]})x_{k,l_k}\Big)^{p_i'-1}>0,
 \end{equation*}
where we have used Euler's theorem for homogeneous functions in the second equality.
$\eqref{pos1}\Rightarrow\eqref{pos2}$ is obvious.
$\eqref{pos2}\Rightarrow\eqref{pos3}$: Let $(i,l_i)\in\I^{\bsig}$, then $0<F^{(\bsig,\bp)}_{i,l_i}(\ones)$. The claim follows from $F^{(\bsig,\bp)}_{i,l_i}(\ones)=\big(\sum_{t=1, t\neq s_i}^m\sum_{j_t=1}^{N_t}T_{j_1,\ldots,j_m}\big)^{p_i'-1}$, where $j_{s_i}=l_i$. 
$\eqref{pos3}\Rightarrow\eqref{pos0}$: Let $(i,l_i)\in\I^{\bsig}$. There exists $j_1,\ldots,j_m$ such that $T_{j_1,\ldots,j_m}>0$ and $j_{s_i}=l_i$. If $d>1$, then $((i,l_i),(k,j_{s_k}))\in\E^{\bsig}(T)$ for $k\neq i$, and if $d=1$, then $i=1$ and $((i,l_i),(k,j_{2}))\in\E^{\bsig}(T)$. In either cases, the $(i,l_i)$-th row of $M$ has at least one positive entry which concludes the proof.
%
\end{proof}

Note that a direct consequence of Lemma \ref{characpos} \eqref{pos3} implies that the $\bsig$-strict nonnegativity property is preserved by the shape partitions' partial ordering of Definition \ref{def:partial_order}. We state this observation in the next lemma, whose straightforward proof is omitted for brevity.
 \begin{lem}\label{strinegincl}
 Let $\bsig=\{\sigma_i\}_{i=1}^{d}$, $\tilde \bsig=\{\tilde\sigma_i\}_{i=1}^{\tilde d}$ be \dimensionalpartition{}s of $T$ such that $\bsig\sqsubseteq \tilde\bsig$. If $T$ is $\bsig$-strictly nonnegative, then it is $\tilde\bsig$-strictly nonnegative.
 \end{lem}
 

Before  concluding this subsection  we want to stress that $\bsig$-strict nonnegativity is a very mild condition as it still allows $T$ to be very sparse. This is illustrated by the following:
 \begin{ex}\label{weaksneg}
 Let $T\in\R_+^{n\times n\times \ldots \times n}$ be an $m$-th order tensor so that $T_{j_1,\ldots,j_m}>0$ if and only if $j_1=\ldots=j_m$. Then, for any \dimensionalpartition{} $\bsig=\{\sigma_i\}_{i=1}^d$ of $T$, $F^{(\bsig,\bp)}$ satisfies $F^{(\bsig,\bp)}(\bx)\in\kone^{\bsig}_{++}$ for every $\bx\in\kone^{\bsig}_{++}$. Note that this tensor has $n$ positive entries and $n^m-n$ zero entries.
 \end{ex}

 \subsection{$\bsig$-weak irreducibility}\label{wirrsec}
Lemmas \ref{characgraph} and \ref{characpos} imply that if $T$ is $\bsig$-weakly irreducible, then $T$ is $\bsig$-strictly nonnegative. Furthermore, Lemma \ref{characgraph} implies that $T$ is $\bsig$-weakly irreducible if and only if $\G^{\bsig}(T)$ is strongly connected.
The lemma below shows that, when $T$ is partially symmetric with respect to $\bsig$, $\G^{\bsig}(T)$ is undirected and, as for $\bsig$-strict nonnegativity,  $\bsig$-weak irreducibility is preserved by the shape partitions' partial order. 
\begin{lem}\label{wirrimpl}
Let $\bsig=\{\sigma_i\}_{i=1}^{d}$, $\tilde \bsig=\{\tilde\sigma_i\}_{i=1}^{\tilde d}$ be \dimensionalpartition{}s of $T$. Then
\begin{enumerate}
\item  if $T$ is $\bsig$-symmetric $\G^{\bsig}(T)$ is undirected, and
\item  if $\bsig\sqsubseteq \tilde \bsig$, $T$ is $\tilde\bsig$-symmetric and $T$ is $\bsig$-weakly irreducible, then $T$ is $\tilde \bsig$-weakly irreducible. 
\end{enumerate}
 \end{lem}
 \begin{proof}
 Let $\bsig=\{\sigma_i\}_{i=1}^d$ and let $(k,l_k),(i,t_i)\in\I^{\bsig}$ be such that $\big((k,l_k),(i,t_i)\big)\in\E^{\bsig}(T)$. If $(k,l_k)= (i,t_i)$ then clearly $\big((i,t_i),(k,l_k)\big)\in\E^{\bsig}(T)$. If $(k,l_k)\neq (i,t_i)$, there exists $j_1,\ldots,j_m$ with $T_{j_1,\ldots,j_m}>0$, $j_{s_k}=l_k$ and $t_i\in\{j_a\mid a\in\sigma_i\}$. Let $s'_i\in\sigma_i$ be such that $t_i=j_{s'_i}$. As $T$ is partially symmetric with respect to $\sigma_i$, we have $T_{j'_1,\ldots,j'_m}=T_{j_1,\ldots,j_m}>0$ where $j'_{s_i}=j_{s'_i}$, $j'_{s'_i}=j_{s_i}$ and $j'_a = j_a$ otherwise. In particular, this implies that $\big((i,t_i),((k,l_k)\big)\in\E^{\bsig}(T)$ and thus $\G^{\bsig}(T)$ is undirected.\newline
 Now, assume that $T$ is $\bsig$-weakly irreducible and let $\tilde \bsig =\{\tilde\sigma_i\}_{i=1}^{\tilde d}$ be a \dimensionalpartition{} of $T$ such that $\bsig\sqsubseteq\tilde \bsig$. Let $\emptyset\neq\tilde V^1,\tilde V^2\subset \I^{\tilde \bsig}$ be such that $\tilde V^1\cap \tilde V^2 = \emptyset$ and $\tilde V^1 \cup \tilde V^2=\I^{\bsig}$. We show that there is an edge between $\tilde V^1$ and $\tilde V^2$ in order to conclude that $T$ is $\tilde \bsig$-weakly irreducible. As 
 $\bsig\sqsubseteq \tilde\bsig$, there exists $g\colon [d]\to [\tilde 
 d]$ such that $\sigma_i\subset \tilde\sigma_{g(i)}$ for all $i\in [d]$. For 
 $k=1,2$, $i\in [\tilde d]$ and $j\in[d]$ let $\tilde V^k_i=\{t_i\mid (i,t_i)\in 
 \tilde V^k\}$ and $V^k_j = \tilde V^k_{g(j)}$. 
 Furthermore, set $V^k = \cup_{j=1}^d \{j\}\times V_j^k$ for $k=1,2$. 
 Then $V^1,V^2$ forms a partitioning of $\I^{\bsig}$ into nonempty 
 disjoints subsets. As $\G^{\bsig}(T)$ is strongly connected, there exists 
 $(k,l_k)\in V^1$ and $(i,t_i)\in V^2$ such that 
 $\big((k,l_k),(i,t_i)\big)\in \E^{\bsig}(T)$. We claim that 
 $\big((g(k),l_k),(g(i),t_i)\big)\in \E^{\tilde\bsig}(T)$, as 
 $(g(k),l_k)\in\tilde V^1$ and $(g(i),t_i)\in\tilde V^2$, this will 
 conclude the proof. Let $\tilde s_{i}=\min\{a\mid a\in\tilde \sigma_i\}$ for $i\in[d]$. There exists $j_1,\ldots,j_m$ 
 such that $T_{j_1,\ldots,j_m}>0$, $j_{s_k}=l_k$ and either 
 $(k,l_k)=(i,t_i)$ and $t_i\in\{j_a \mid a\in\sigma_i\setminus\{s_i\}\}$ 
 or $(k,l_k)\neq (i,t_i)$ and $t_i\in\{j_a \mid a\in\sigma_i\}$. In either 
 cases, one can use $\tilde \bsig$-symmetry of $T$ and rearrange the 
 $j_1,\ldots,j_m$ into $j_1',\ldots,j_m'$ so that 
 $T_{j_1',\ldots,j_m'}=T_{j_1,\ldots,j_m}>0$, $j'_{\tilde s_{g(k)}}=j_{s_k}$, $j'_{s_k}= 
 j_{\tilde s_{g(k)}}$ and $j_a' = j_a$ for all 
 $a\in[m]\setminus\{s_k,\tilde s_{g(k)}\}$. In particular this implies 
 our claim and the proof is done.
 \end{proof}

Note that  the $\bsig$-symmetry assumption is essential in order to have property 2 of Lemma \ref{wirrimpl}. In fact, for instance, if $\bsig^1,\bsig^2,\bsig^3$ are defined as in \eqref{shapex}, then  the tensor of Example \ref{graphex} is $\bsig^{i}$-weakly irreducible for $i=1,3$ but is not $\bsig^{2}$-weakly irreducible. In fact, already for the case of ${3\times 3 \times 3}$ tensors, for any $\Omega\subset \{1,2,3\}$ one can find a tensor which is $\bsig^i$-weakly irreducible for $i\in\Omega$ and not $\bsig^i$-weakly irreducible for $i\in \{1,2,3\}\setminus \Omega$. We prove the latter statement in the following remark where we exhibit $3\times 3\times 3$ tensors with binary entries which are $\bsig^i$-weakly irreducible for $i\in\Omega$ and not $\bsig^i$-weakly irreducible for $i\in \{1,2,3\}\setminus \Omega$, for all $\Omega\subset \{1,2,3\}$. As $\{1,2,3\}$ has $8$ different subsets $\Omega$, for the sake of brevity, we simply list all entries of these tensors in the reverse lexicographic order as a binary string of length 27. So, for instance, the tensor $T$ of Example \ref{graphex} can be compactly written as

\begin{equation}\label{binarynotation}
T\equiv \!\!\!\!\overset{\overset{T_{1,1,1}}{\downarrow}}{0} \,   \overset{\overset{T_{2,1,1}}{\downarrow}}{0} \,    \overset{\overset{T_{3,1,1}}{\downarrow}}{0}   \, \overset{\overset{T_{1,2,1}}{\downarrow}}{0}   \, \overset{\overset{T_{2,2,1}}{\downarrow}}{1}   \,  \overset{\overset{T_{3,2,1}}{\downarrow}}{1} \,    \overset{\overset{T_{1,3,1}}{\downarrow}}{1} \!\!\!\!     0     0    
\overset{\overset{\cdots}{{\color{white}\downarrow}}}{0}     0     0     
\overset{\overset{\cdots}{{\color{white}\downarrow}}}{0}     1     0     
\overset{\overset{\cdots}{{\color{white}\downarrow}}}{0}     0     0   \!\!\!\!  \overset{\overset{T_{1,1,3}}{\downarrow}}{1}  \!\!\!\!   0     0     0     \overset{\overset{\cdots}{{\color{white}\downarrow}}}{0}    0     0     0      \!\!\!\! \overset{\overset{T_{3,3,3}}{\downarrow}}{0}
\end{equation}

 \begin{rmq}\label{sirrcex}
 Let $\bsig^1,\bsig^2,\bsig^3$ be as in \eqref{shapex}. For every $\epsilon_1,\epsilon_2,\epsilon_3\in\{0,1\}$ there exists a tensor $ T^{(\epsilon_1,\epsilon_2,\epsilon_3)}\in\{0,1\}^{3\times 3\times 3}$ such that for $i=1,2,3$, $T^{(\epsilon_1,\epsilon_2,\epsilon_3)}$ is $\bsig^i$-weakly irreducible if $\epsilon_i=1$ and $T^{(\epsilon_1,\epsilon_2,\epsilon_3)}$ is not $\bsig^i$-weakly irreducible if $\epsilon_i=0$. Precisely, with the notation introduced in \eqref{binarynotation} we have
 \begin{equation*}
 \begin{array}{l l l l}
\!\! T^{(0,0,0)}\! \equiv\! 0     0     0     0     0     0     0     0     0     0     0     0     0     0     0     0     0     0     0     0     0     0     0     0     0     0     0,&\!\!\!\!
T^{(1,0,0)}\!\equiv\! 0     1     1     1     0     0     1     0     0     0     0     0     0     0     0     0     0     0     0     0     0     0     0     0     0     0     0, \\
\!\! T^{(0,1,0)}\!\equiv\!  1     1     1     0     1     0     1     0     0     0     0     0     0     0     0     0     0     0     0     0     0     0     0     0     0     0     0,&\!\!\!\!
 T^{(1,1,0)}\!\equiv\! 1     1     1     1     0     0     1     0     0     0     0     0     0     0     0     0     0     0     0     0     0     0     0     0     0     0     0, \\
\!\! T^{(0,0,1)}\!\equiv\! 0     0     1     0     1     0     1     0     0     0     0     0     0     1     0     0     0     0     1     0     0     0     0     0     0     0     0,&\!\!\!\!
T^{(1,0,1)}\!\equiv\!  0     0     0     0     1     1     1     0     0     0     0     0     0     1     0     0     0     0     1     0     0     0     0     0     0     0     0, \\ 
\!\! T^{(0,1,1)}\!\equiv\!  0     0     1     0     1     0     1     0     0     0     1     0     0     0     0     0     0     0     1     0     0     0     0     0     0     0     0,&\!\!\!\! 
T^{(1,1,1)}\!\equiv\! 1     1     1     1     0     0     1     0     0     1     0     0     0     0     0     0     0     0     1     0     0     0     0     0     0     0     0.
 \end{array}
 \end{equation*}
All the tensors given above are  not $\bsig^i$-strongly irreducible,  for $i=1,2,3$.
\end{rmq}

 \subsection{$\bsig$-strong irreducibility}\label{irrsec}
 We characterize $\bsig$-strong irreducibility and discuss its connection with $F^{(\bsig,\bp)}$. 
In particular, we prove that $T$ is $\bsig$-strongly irreducible if and only if for every $\bx^0\in\kone^{\bsig}_{+,0}$, there exists an integer $N$ such that $\bx^{N}\in\kone^{\bsig}_{++}$ where $\bx^{k+1}=\bx^k+F^{(\bsig,\bp)}(\bx^k)$ for $k=0,1,\ldots,N$ Note that this is conceptually analogous to the well--known property of nonnegative matrices $Q \in \R^{n\times n}$ for which there exists an integer $N\leq n$ such that $(I+Q)^N$ is positive.   This property implies that every nonnegative $(\bsig,\bp)$-eigenvector of $T$ is strictly positive. Indeed, the following lemma holds: 
\begin{lem}\label{posvectirr}
  Suppose that for every $\bx^0\in\kone^{\bsig}_{+,0}\setminus\kone^{\bsig}_{++}$, there exists $N$ such that $\bx^{N}\in\kone^{\bsig}_{++}$ where $\bx^{k+1}=\bx^k+F^{(\bsig,\bp)}(\bx^k)$ for $k=0,1,\ldots$. Then, for every $(\bt,\bx)\in\R^d_{+}\times\kone^{\bsig}_{+,0}$ such that $F^{(\bsig,\bp)}(\bx)=\bt\krog\bx$, we have $\bx\in\kone^{\bsig}_{++}$. 
 \end{lem}
 \begin{proof}
 If $\bx\in\kone_{++}^{\bsig}$, there is nothing to prove so let us assume that $\bx\notin\kone_{++}^{\bsig}$. Set $\bx^0=\bx$ and let $N$ be such that $\bx^{N}\in\kone^{\bsig}_{++}$. Note that for every $k$ we have $\bx^k = \bdel^{(k)}\krog\bx$ where $\bdel^{(k)}\in\R^d_{++}$ is given by $\bdel^{(0)}=\ones$ and $\bdel^{(j+1)}=\bdel^{(j)}+(\bdel^{(j)})^A\circ\bt$ for all $j=0,1,\ldots$ In particular $\bdel^{(N)}\in\R^d_{++}$ and so $\bdel^{(N)}\krog\bx=\bx^{N}\in\kone^{\bsig}_{++}$ implies that $\bx\in\kone_{++}^{\bsig}$ which concludes the proof.
 \end{proof}

The lemma below gives several equivalent characterizations of $\bsig$-strong irreducibility:
 \begin{lem}\label{characirr}
 The following statements are equivalent:
 \begin{enumerate}[(i)]
 \item $T$ is $\bsig$-strongly irreducible.\label{defirrlem}
 \item For every $\emptyset \neq V\subset\I^{\bsig}$ such that $V_i=\{l_i\in[n_i]\mid (i,l_i)\in V\}\neq [n_i]$ for all $i\in [d]$,  the following holds: There exists $k\in[d]$ and $j_1,\ldots,j_m$ such that $T_{j_1,\ldots,j_m}>0$, $j_{s_k}\in V_k$, $j_{t}\in [n_k]\setminus V_k$, $t\in\sigma_k\setminus\{s_k\}$ and $j_t\in [n_i]\setminus V_{i}$, $t\in \sigma_i, i \in[d]\setminus\{k\}$.\label{lemsigirr}
  \item There exists $N\leq n_1+\ldots+n_d-d$ such that for all $\bj=(j_1,\ldots,j_d)\in [n_1]\times\ldots\times [n_d]$, it holds $\be_{\bj}^{N} \in\kone_{++}^{\bsig}$, where $\be_{\bj}^{k+1}=\be_{\bj}^k+G^{(\bsig)}(\be_{\bj}^k)$, $G^{(\bsig)}\colon\kone_{+}^{\bsig}\to \kone_{+}^{\bsig}$ is given by $G^{(\bsig)}(x)=(\T_{s_1}(x^{[\bsig]}),\ldots,\T_{s_d}(x^{[\bsig]}))$
  and $(\be_{\bj}^{0})_{k,l_k}=1$ if $l_k=j_k$, $(\be_{\bj}^{0})_{k,l_k}=0$ else.\label{characirrG} 
 \item  For every $\bx^0\in\kone_{+,0}^{\bsig}$, there exists a positive integer $N_{\bx}$ such that $\bx^{N_{\bx}}\in\kone^{\bsig}_{++}$ where $\bx^{k+1}=\bx^k+F^{(\bsig,\bp)}(\bx^k)$ and $k = 0,1,2\ldots$\label{FIRR}
\item $Q\big(F^{(\bsig,\bp)}(\bz)\big)\not\supset Q(\bz)$ for every $\bz\in\kone_{+,0}^{\bsig}\sauf\kone_{++}^{\bsig}$, where, for every $\bx\in\kone_{+}^{\bsig}$, $Q(\bx)=\{(i,j_i)\in\I^{\bsig} \ |\ \ x_{i,j_i}=0\}$.\label{characirrQ}
 \end{enumerate}
 \end{lem}
 \begin{proof}
 Note that the equivalence \eqref{defirrlem}$\Leftrightarrow $\eqref{characirrQ} is direct. We show the other implications by a circular argument, i.e. \eqref{lemsigirr}$\Rightarrow\ldots\Rightarrow$\eqref{characirrQ}$\Rightarrow$\eqref{lemsigirr}.\newline
\eqref{lemsigirr} $\Rightarrow$ \eqref{characirrG}: Let $\bz\in\kone^{\bsig}_{+,0}\setminus\kone^{\bsig}_{++}$ and let $Q(\cdot)$ be defined as in \eqref{characirrQ}. Let $V=Q(\bz)$, then  $V\neq \emptyset$ and $V_i=\{j_i\mid z_{i,j_i}=0\}\neq [n_i]$ for all $i\in[d]$. Now, there exists $k\in[d]$ and $j_1,\ldots,j_m$ such that $T_{j_1,\ldots,j_m}>0$, $j_{s_k}\in V_k$, $j_{t}\in [n_k]\setminus V_k$, $t\in\sigma_k\setminus\{s_k\}$ and $j_t\in [n_i]\setminus V_{i}$, $t\in \sigma_i, i \in[d]\setminus\{k\}$. It follows that
\begin{equation*}
G_{k,j_{s_k}}^{(\bsig)}(z)=\T_{s_k,j_{s_k}}(\bz^{[\bsig]})\geq T_{j_1,\ldots,j_m}\Big(\prod_{t\in\sigma_k, t\neq s_k}z_{k,j_t}\Big)\prod_{i=1, i \neq k}^d\prod_{t\in\sigma_i}z_{i,j_t}>0,
\end{equation*}
and so $(k,j_{s_k})\in Q(\bz)\setminus Q(G^{(\bsig)}(\bz^{[\bsig]}))$. This shows that $|Q(\bz)|>|Q(\bz+G^{(\bsig)}(\bz^{[\bsig]}))|$ for all $\bz\in\kone^{\bsig}_{+,0}\setminus\kone^{\bsig}_{++}$. It follows that for all $\bj\in[n_1]\times\ldots\times [n_d]$, we have $|Q(\be^{k}_{\bj})| > |Q(\be^{k+1}_{\bj})|$ for all $k$ such that $|Q(\be^{k}_{\bj})|>0$. Finally, note that if $\be^{k}\in\kone^{\bsig}_{++}$, then $\be^{k+l}\in\kone^{\bsig}_{++}$ for all $l\geq 0$ and $|Q(\be^{0}_{\bj})|\leq n_1+\ldots+n_d-d$, so that $|Q(\be^{n_1+\ldots+n_d-d}_{\bj})|=0$ for all $\bj\in[n_1]\times\ldots\times [n_d]$ which concludes this part.\newline
\eqref{characirrG}$\Rightarrow$\eqref{FIRR}: Let $\bx^{0}\in\kone^{\bsig}_{+,0}\setminus\kone^{\bsig}_{++}$, then there exists $\bj\in[n_1]\times\ldots\times [n_d]$ and $\bdel^{(0)}\in \R^d_{++}$ such that $\bdel^{(0)}\krog \be^{0}_{\bj}\leq \bx^{0}$. We prove by induction that for every $k$ there exists $\bdel^{(k)}\in\R^d_{++}$ such that $\bdel^{(k)}\krog \be^{k}_{\bj}\leq \bx^{k}$. The case $k=0$ is discussed above, so suppose it is true for a $k\geq 0$ and let $\bdel^{(k)}\in\R^d_{++}$ be such that $\bdel^{(k)}\krog \be^{k}_{\bj}\leq \bx^{k}$. Set
\begin{equation*}
\alpha_i =\min\big\{\Big(G^{(\bsig)}_{i,j_i}\big((\be^{k}_{\bj})^{[\bsig]}\big)\Big)^{p_i'-2}\ \big|\ j_i\in[n_i]\text{ and }G^{(\bsig)}_{i,j_i}\big((\be^{k}_{\bj})^{[\bsig]}\big)>0\big\}\qquad \forall i \in[d],
\end{equation*}
and let $\delta^{(k+1)}_i=\min\{\delta_i^{(k)},((\bdel^{(k)})^A)_i\alpha_i \}>0$ for $i\in[d]$. Then, as $F^{(\bsig,\bp)}$ is order-preserving, we have
\begin{align*}
\bdel^{(k+1)}\krog \be^{k+1}_{\bj} &\lek \bdel^{(k)}\krog \be^{k}_{\bj}+\big(\bal \circ (\bdel^{(k)})^A\big) \krog G^{(\bsig)}\big((\be^{k}_{\bj})^{[\bsig]}\big)\\
&\lek \bdel^{(k)}\krog \be^{k}_{\bj}+F^{(\bsig,\bp)}(\bdel^{(k)}\krog\be^k_{\bj})\lek \bx^k+ F^{(\bsig,\bp)}(\bx^k)=\bx^{k+1}.
\end{align*}
This concludes our induction proof. In particular, we have $0\lekkk\bdel^{(N)}\krog \be^{N}_{\bj}\leq \bx^{N}$ for all $N\geq n_1+\ldots+n_d$ which shows the claim.\newline
 \eqref{FIRR}$\Rightarrow$\eqref{characirrQ}: We show that if \eqref{characirrQ} does not hold, then \eqref{FIRR} does not hold either. Note that for $\bx,\by\in\kone^{\bsig}_{+,0}$, if $Q(\bx)=Q(\by)$, then there exists $\bal,\bbe\in\R^d_{++}$ such that $\bal\krog\by \lek \bx \lek \bbe\krog\by$ which implies that $Q\big(F^{(\bsig,\bp)}(\bx)\big)=Q\big(F^{(\bsig,\bp)}(\by)\big)$ as we then have $\bal^A\krog F^{(\bsig,\bp)}(\by) \lek F^{(\bsig,\bp)}(\bx) \lek \bbe^A\krog F^{(\bsig,\bp)}(\by)$. Now, suppose that there exists $\bx^0\in\kone_{+,0}$ with $\emptyset \neq Q(\bx^0)\subset Q(F^{(\bsig,\bp)}(\bx^0))$. Then, we have
  \begin{equation*}
  Q(\bx^1)=Q\big(\bx^0 + F^{(\bsig,\bp)}(\bx^0)\big)=Q(\bx^0)\cap Q\big(F^{(\bsig,\bp)}(\bx^0)\big) =Q(\bx^0).
  \end{equation*}
  Using induction and the arguments above, if $Q\big(\bx^{k}\big)=Q(\bx^0)$ for $k>0$, then 
  \begin{equation*}
  Q\big(\bx^{k+1}\big)=Q\big(F^{(\bsig,\bp)}(\bx^k)\big)\cap Q\big(\bx^k\big)=Q\big(F^{(\bsig,\bp)}(\bx^0)\big)\cap Q\big(\bx^k\big)=Q(\bx^0). 
  \end{equation*}
  Hence, $Q\big(\bx^{k}\big)\neq \emptyset$ for every $k>0$ and thus \eqref{FIRR} can not be satisfied.\newline
 \eqref{characirrQ}$\Rightarrow$\eqref{lemsigirr}: Let $\emptyset \neq V \subset \I^{\bsig}$ be such that $V_i=\{j_i\mid (i,j_i)\in V\}\neq [n_i]$ for all $i$. Define $\bz\in\kone_{+}$ as $z_{i,j_i}=0 $ if $(i,j_i)\in V$ and $z_{i,j_i}=1 $ else. Then $\bz\in \kone^{\bsig}_{+,0}$ as $V_i\neq [n_i]$ for all $i$, and $\bz\notin \kone^{\bsig}_{++}$ as $V\neq \emptyset$. 
 Now, we have $Q(F^{(\bsig,\bp)}(\bz))\not\subset Q(\bz)$ and so there exists $(k,l_k)\in \I^{\bsig}$ such that $F^{(\bsig,\bp)}_{k,l_k}(\bz)>0$ and $x_{k,l_k}=0$. $F^{(\bsig,\bp)}_{k,l_k}(\bz)>0$ implies the existence of $j_1,\ldots,j_m$ such that 
 $$T_{j_1,\ldots,j_m}\Big(\prod_{t\in\sigma_k, t\neq s_k}z_{k,j_t}\Big)\prod_{i=1, i \neq k}^d\prod_{t\in\sigma_i}z_{i,j_t}>0.$$
 Hence, we have $T_{j_1,\ldots,j_m}>0$ and $z_{i,j_t}>0$ for all $t\in \sigma_i, i \neq k$ and $z_{k,j_t}>0$, $t\in\sigma_k\setminus\{s_k\}$. As $(k,l_k)\in V$ and $z_{i,j_t}>0$ implies $(i,j_t)\notin V$, this concludes the proof.
 \end{proof}
 
Let us point out that the second characterization in the above lemma reduces to the definition of irreducibility introduced for the cases $d=1$ and $d=m$ in \cite{Chang} and \cite{Fried}, respectively. Furthermore, the third characterization is particularly relevant as it allows to introduce a simple algorithm for checking $\bsig$-irreducibility. In particular,  observe that, when $d=1$, such characterization  reduces to Theorem 5.2 of \cite{Yang2}.

Finally, with the next lemma we prove that, as for $\bsig$-strict nonnegativity and $\bsig$-weak irreducibility, for $\bsig$-symmetric and nonnegative tensors, $\bsig$-strong irreducibility is preserved by the partial order on shape partitions. 

 \begin{lem}\label{sirrimpl}
 	Let $\bsig$, $\tilde \bsig$ be \dimensionalpartition{}s of $T$. If $T$ is $\tilde \bsig$-symmetric, $\bsig$-strongly irreducible and $\bsig\sqsubset \tilde \bsig$, then $T$ is $\tilde\bsig$-strongly irreducible.
 \end{lem}
 \begin{proof}
By Lemma \ref{characirr}, we may assume without loss of generality that $p_i=\tilde p_j = 2$ for all $i\in[d], j \in [\tilde d]$.  Now, there exists $\bar \bsig$ such that $\bx^{[\bar \bsig]}\in\kone_{+}^{\bsig}$ for all $\bx\in\kone_{+}^{\tilde \bsig}$. Let $\bx^0\in\kone_{+,0}^{\tilde\bsig}\setminus\kone_{++}^{\tilde\bsig}$ and, for $k\in\N$, define $\bx^{k+1}=\bx^k+F^{(\tilde\bsig,\bp)}(\bx^k)$. We show that $\bx^K\in\kone_{++}^{\tilde \bsig}$ for some $K>0$ so that the claim follows from Lemma \ref{characirr}. Define $\bz^0=(\bx^0)^{[\bar\bsig]}\in\kone_{+,0}^{\bsig}$ and $\bz^{k+1}=\bz^k+F^{(\bsig,\bp)}(\bz^k)$ for $k\in\N$. As $T$ is $\tilde\bsig$-symmetric, we have $\bz^k=(\bx^k)^{[\bar\bsig]}$ for all $k$. Lemma \ref{characirr} implies the existence of $K>0$ such that $\bz^K\in\kone_{++}^{\bsig}$ and thus $\bx^K\in\kone_{++}^{\tilde\bsig}$ which conclude the proof.
\end{proof}
As for the case of $\bsig$-weak irreducibility, we show in the following remark that $\bsig$-symmetry is an essential requirement for the  above lemma.
\begin{rmq}\label{stirrcex}
Let $\bsig^1,\bsig^2,\bsig^3$ be as in \eqref{shapex}. For every $\epsilon_1,\epsilon_2,\epsilon_3\in\{0,1\}$, there exists a tensor $T^{(\epsilon_1,\epsilon_2,\epsilon_3)}\in\{0,1\}^{3\times 3\times 3}$ such that for $i=1,2,3$, $T^{(\epsilon_1,\epsilon_2,\epsilon_3)}$ is $\bsig^i$-strongly irreducible if $\epsilon_i=1$ and $T^{(\epsilon_1,\epsilon_2,\epsilon_3)}$ is not $\bsig^i$-strongly irreducible if $\epsilon_i=0$. Precisely, with the notation of \eqref{binarynotation} we have
\begin{equation*}
\begin{array}{l l}
\!\! T^{(0,0,0)}\!\equiv\! 1     1     1     1     0     0     1     0     0     1     0     0     0     0     0     0     0     0     1     0     0     0     0     0     0     0     0,&\!\!\!\!
T^{(1,0,0)}\!\equiv\! 1     1     1     0     0     0     0     0     0     0     0     0     1     0     0     0     0     0     0     0     0     0     0     0     1     0     0,\\ 
\!\! T^{(0,1,0)}\!\equiv\! 1     0     0     0     1     0     1     0     0     1     1     1     0     0     0     0     0     0     1     1     1     0     0     0     0     0     0,&\!\!\!\!
T^{(1,1,0)}\!\equiv\! 1     1     1     1     0     0     1     0     0     1     0     0     1     0     0     0     0     0     1     0     0     0     0     0     1     0     0,\\ 
\!\! T^{(0,0,1)}\!\equiv\! 1     1     0     0     1     1     0     1     1     1     0     0     0     0     0     0     0     0     1     0     0     0     0     0     0     0     0,&\!\!\!\!
T^{(1,0,1)}\!\equiv\! 0     1     1     1     1     1     1     1     1     0     0     0     1     0     0     0     0     0     0     0     0     0     0     0     1     0     0,\\ 
\!\! T^{(0,1,1)}\!\equiv\! 1     0     0     1     0     0     1     0     0     1     1     1     0     0     0     0     0     0     1     1     1     0     0     0     0     0     0,&\!\!\!\!
T^{(1,1,1)}\!\equiv\! 1     1     0     0     1     1     0     1     1     1     0     0     1     0     0     0     0     0     1     0     0     0     0     0     1     0     0.
\end{array}
\end{equation*}
All the tensors given above are $\bsig^i$-weakly irreducible, for $i=1,2,3$. 

Furthermore, in the case $d=2$, it follows from Theorem 2.4 in \cite{Qi_rect_1} that if $T$ is irreducible in the sense of Definition 1 in \cite{Chang_rect_eig}, then $T$ is $\bsig$-strongly irreducible. However, the converse is not true. For example,  the tensor $T^{(0,1,0)}$ of Remark \ref{stirrcex} is $\bsig^2$-strongly irreducible, but, with $\bx=\big((1,0,0)^\top,(0,1,0)^\top\big)$, we have $\T(\bx^{[\bsig^2]})=((0,0,0)^\top,(1,0,0)^\top)$ and so, by Lemma 2 in \cite{Chang_rect_eig}, $T$ can not be irreducible in the sense of Definition 1 of \cite{Chang_rect_eig}.
\end{rmq}
%
%
%
%
%
%
\subsection{Proof of Theorem \ref{irrthm}}
Note that points  \eqref{irrthm3} and \eqref{irrthm4} of Theorem \ref{irrthm} follow immediately from Lemmas \ref{wirrimpl} and \ref{sirrimpl}, respectively. Thus, we only need to prove that $\bsig$-strong irreducibility implies $\bsig$-weak irreducibility and that $\bsig$-weak irreducibility implies $\bsig$-strict nonnegativity. This is addressed by the following lemma. 

For completeness, let us remark that in the particular cases $d=1$ and $d=m$, it is known that (strong) irreducibility implies weak irreducibility (see Lemma 3.1 in \cite{Fried}). Furthermore, still for the particular cases $d=1$ and $d=m$, it was proved in Proposition 8, (b) of \cite{us} and Corollary 2.1. of \cite{Hu2014} that weak irreducibility implies strict nonnegativity. All these results are particular cases of the following:
\begin{lem}\label{irrsowirr}
If $T$ is $\bsig$-strongly irreducible, then it is $\bsig$-weakly irreducible. If $T$ is $\bsig$-weakly irreducible, then it is $\bsig$-strictly nonnegative.
\end{lem}
\begin{proof}
The case $d=1$ follows from Corollary 2.1. \cite{Hu2014} and Lemma 3.1 \cite{Fried}. Now, suppose $d>1$. Let $M$ be as in definition \eqref{defirr}. If $M$ is irreducible, then $M$ has at least one nonzero entry per row and thus if $T$ is $\bsig$-weakly irreducible, then $T$ is $\bsig$-strict nonnegative.
Now, suppose that $T$ is $\bsig$-strongly irreducible and let us show that $T$ is $\bsig$-weakly irreducible. To this end, we first show that $T$ is $\bsig$-strictly nonnegative. Suppose by 
contradiction that it is not the case. By Lemma \ref{characpos}, there exists $\bx\in\kone_{++}^{\bsig}$ and 
$(k,l_k)\in\I^{\bsig}$ such that $F^{(\bsig,\bp)}_{k,l_k}(\bx)=0$. Let $\bz\in\kone_{+,0}^{\bsig}$ be defined as 
$z_{i,j_i}=x_{i,j_i}$ for all $(i,j_i)\in\I^{\bsig}\setminus\{(k,l_k)\}$ and $z_{k,l_k}=0$. Then, as $\bz\lek \bx$, we 
have $F_{k,l_k}^{(\bsig,\bp)}(\bz)\leq F^{(\bsig,\bp)}_{k,l_k}(\bx)=0$ which contradicts Lemma \ref{characirr}, 
\eqref{characirrQ}.

Now, to show that $T$ is $\bsig$-weakly irreducible, we show that for every nonempty subsets $V^1,V^2 \subset 
\I^{\bsig}$ with $V^1\cap V^2 = \emptyset$ and $V^1 \cup V^2 = \I^{\bsig}$ there exists $(k,l_k)\in V^1$ and 
$(i,t_i)\in V^2$ such that $\big((k,l_k),(i,t_i)\big)\in \E^{\bsig}(T)$. So let $V^1, V^2$ be such a partition of 
$\I^{\bsig}$ and set $V^j_i = \{t_i\in [n_i]\mid (i,t_i)\in V^j\}$ for $i\in[d]$, $j=1,2$. First, assume that 
$V^1_i\neq [n_i]$ for all $i\in [d]$. Then, as $T$ is $\bsig$-strongly irreducible, by Lemma \ref{characirr}  \eqref{lemsigirr}, there exists $k\in[d]$ and 
$j_1,\ldots,j_m$ such that $T_{j_1,\ldots,j_m}>0$, $j_{s_k}\in V^1_k$, $j_{t}\in V^2_k$, 
$t\in\sigma_k\setminus\{s_k\}$ and $j_t\in V^2_{i}$, $t\in \sigma_i, i \in[d]\setminus\{k\}$. It follows that 
$\big((k,j_{s_k}),(i,j_{s_i})\big)\in\E^{\bsig}(T)$ for all $i\neq k$ and we are done. Now, suppose that there exists $\bar k\in [d]$ 
such that $V^1_{\bar k}= [n_{\bar k}]$. We claim that if there is no edge between $V^1$ and $V^2$ in 
$\G^{\bsig}(T)$, then $T$ is not $\bsig$-strictly nonnegative which contradicts our previous argument. Indeed, 
suppose that $\big((k,l_k),(i,t_i)\big)\notin \E^{\bsig}(T)$ for all $(k,l_k)\in V^1$ and $(i,t_i)\in V^2$. Let 
$(i,t_i)\in V^2$. Note that $i\neq \bar k$ as $V^{1}_{\bar k}=[n_{\bar k}]$. Furthermore, we have 
$T_{j_1,\ldots,j_m}=0$ for all $j_1,\ldots,j_m$ such that $j_{s_i}=t_i$ and $j_{s_{{\bar k}}}\in [n_{\bar k}]$. By 
Lemma \ref{characpos}, \eqref{pos3}, this implies that $T$ is not $\bsig$-strictly nonnegative, a contradiction. 
Thus, there exists $(k,l_k)\in V^1$ and $(i,t_i)\in V^2$ such that $\big((k,l_k),(i,t_i)\big)\in \E^{\bsig}(T)$ and as 
this is true for every partition of $\I^{\bsig}$, it follows that $\G^{\bsig}(T)$ is connected.
\end{proof}

We finally have all the tools for the proof of Theorem \ref{irrthm}, which is now a simple consequence of what have been discussed so far.
\begin{proof}[Proof of Theorem \ref{irrthm}]
\eqref{irrthm0}, \eqref{irrthm1}, \eqref{irrthm2} follow from Lemma \ref{irrsowirr} and \eqref{irrthm3}, \eqref{irrthm4} follow from Lemmas \ref{wirrimpl} and \ref{sirrimpl}, respectively.
\end{proof}

We conclude the paper by proving the other two main results of Section \ref{mainresults}.
\section{Proof of Theorems \ref{PFtheory} and \ref{PMthm}}\label{proofsec}
Recall that the homogeneity matrix $A$ of $F^{(\bsig,\bp)}$ is given as
$$A=\diag(p_1'-1,\ldots,p_d'-1)(\ones\bnu^\top-I), \qquad \bnu=(|\sigma_1|,\ldots,|\sigma_d|),$$
and $\bb\in\R^d_{++}$ is the unique positive vector such that $A^\top\bb=\rho(A)\bb$ and $\sum_{i=1}^db_i=1$.
For the proof of Theorem \ref{PFtheory}, we first need the following additional lemma.
\begin{lem}\label{connectmax}
Suppose that $\rho(A)\leq 1$ and $T$ is $\bsig$-strictly nonnegative. If $(\bt,\bu)\in\R^d_{+}\times\S_+^{(\bsig,\bp)}$ satisfies $F^{(\bsig,\bp)}(\bu)=\bt \otimes \bu$ with $\prod_{i=1}^d\theta_i^{b_i}=r_{\bb}(F^{(\bsig,\bp)})$, then $\theta_i^{p_i-1} =r^{(\bsig,\bp)}(T)$ for all $i\in[d]$ and $(r^{(\bsig,\bp)}(T),\bu)$ is a $(\bsig,\bp)$-eigenpair of $T$.
\end{lem}
\begin{proof}
By Lemma \ref{characpos}, we have that the $\bsig$-strict nonnegativity of $T$ implies $F^{(\bsig,\bp)}(\bx)\in\kone_{++}^{\bsig}$ for all $\bx\in\kone_{++}^{\bsig}$. Now, as $F^{(\bsig,\bp)}_i(\bu)=\theta_i\bu$ for all $i\in[d]$, Lemma \ref{mhconnect} implies the existence of $\lambda\in\R_+$ such that $\theta_i^{p_i-1}=\lambda$ for all $i\in[d]$ and $(\lambda,\bu)$ is a $(\bsig,\bp)$-eigenpair of $T$. We prove that $\lambda = r^{(\bsig,\bp)}(T)$. Clearly, we have $\lambda \leq r^{(\bsig,\bp)}(T)$. Now, let $(\vartheta,\bv)\in\R\times \R^{n_1}\times\ldots\times \R^{n_d}$ be any $(\bsig,\bp)$-eigenpair of $T$. Then, by definition, we have $\T_{s_i}(\bv^{[\bsig]})=\vartheta\psi_{p_i'}(\bv_i)$ for every $i\in[d]$. It follows that $\psi_{p_i'}(\T_{s_i,j_i}(\bv^{[\bsig]}))=|\vartheta|^{p_i'-1}\sign(\vartheta)\bv_i$ for all $i$. In particular, by the triangle inequality, with $\bw=|\bv|$, i.e. $\bw$ is the component-wise absolute value of $\bv$, we have $|\psi_{p_i'}(\T_{s_i,j_i}(\bv^{[\bsig]}))|\leq F^{(\bsig,\bp)}_{i,j_i}(\bw)$. Hence, for $(i,j_i)\in\I^{\bsig}$ such that $w_{i,j_i}>0$, it holds 
\begin{equation*}
|\vartheta|^{p_i'-1} =\bigg|\frac{\psi_{p_i'}(\T_{s_i,j_i}(\bv^{[\bsig]}))}{v_{i,j_i}}\bigg| \leq \frac{F^{(\bsig,\bp)}_{i,j_i}(\bw)}{w_{i,j_i}}\, .
\end{equation*}
Now, as $\norm{\bv_i}_{p_i}=1$ for all $i\in [d]$, we have $\bw\in\S^{(\bsig,\bp)}_+$ and thus Theorem 6.1 in \cite{mhpfpaper} implies that, with $\gamma'=\sum_{i=1}^db_ip_i'$, it holds $|\vartheta|^{\gamma'-1}=\prod_{i=1}^d |\vartheta|^{b_i(p_i'-1)} \leq r_{\bb}(F^{(\bsig,\bp)})=\lambda^{\gamma'-1}$. Finally, as $\gamma'=\frac{\gamma}{\gamma-1}$, where $\gamma$ is defined as in Lemma \ref{specradlem}, we have $\gamma'>1$ and thus it follows that $|\vartheta|\leq \lambda$ implying that $\lambda \geq r^{(\bsig,\bp)}(T)$ which concludes the proof.
\end{proof}

\begin{proof}[Proof of Theorem \ref{PFtheory}]
Note that, by Lemma \ref{characpos}, we have  $F^{(\bsig,\bp)}(\bx)\in\kone_{++}^{\bsig}$ for all $\bx\in\kone_{++}^{\bsig}$, as $T$ is $\bsig$-strictly nonnegative. 
\begin{enumerate}[(i)]
\item First note that $\gamma \in (1,\infty)$ by Lemma \ref{specradlem}.
To show the existence of a $(\bsig,\bp)$-eigenpair $(\lambda,\bu)\in\R_+\times \kone_{+,0}^{\bsig}$ of $T$ such that $\lambda =r^{(\bsig,\bp)}(T)$, it is enough, by Lemma \ref{connectmax}, to show that there exists $(\bt,\bu)\in\R^d_{+}\times\S_+^{(\bsig,\bp)}$ such that $F^{(\bsig,\bp)}(\bu)=\bt\krog\bu$ and $\prod_{i=1}^d\theta_i^{b_i}=r_{\bb}(F^{(\bsig,\bp)})$. 
If $\rho(A)=1$, the existence of $(\bt,\bu)$ follows from Theorem 4.1 in \cite{mhpfpaper}. 
If $\rho(A)<1$, then Theorem 3.1 in \cite{mhpfpaper}, implies the existence of $(\tilde\bt,\tilde \bu)\in\R^d_{++}\times\S_{++}^{(\bsig,\bp)}$ such that $F^{(\bsig,\bp)}(\tilde \bu)=\tilde \bt\krog\tilde \bu$. As $\tilde \bu$ is positive, Theorem 6.1 in \cite{mhpfpaper} implies that $\prod_{i=1}^d\tilde\theta_i^{b_i}=r_{\bb}(F^{(\bsig,\bp)})$ and thus we can choose $(\bt,\bu)=(\tilde \bt,\tilde \bu)$. In any case, we have proved the existence of $(\bt,\bu)$ with the desired property and it follows from Lemma \ref{connectmax} that $(r^{(\bsig,\bp)}(T),\bu)$ is a $(\bsig,\bp)$-eigenpair of $T$. 
\item
Lemma \ref{specradlem} implies that $r^{(\bsig,\bp)}(T)=r_{\bb}(F^{(\bsig,\bp)})^{\gamma-1}$ and $\gamma \in (1,\infty)$. Thus, \eqref{finalCW} and \eqref{Gelf} follow from Theorems 6.1 and 4.1 in \cite{mhpfpaper}, respectively.
\item First note that as $\rho(A)\leq1$ and $F^{(\bsig,\bp)}(\bx)\in\kone_{++}^{\bsig}$ for all $\bx\in\kone_{++}^{\bsig}$, we know from Theorem 6.1 in \cite{mhpfpaper}, that for any $(\tilde\bt,\tilde \bu)\in\R^d_{++}\times\S_{++}^{(\bsig,\bp)}$ such that $F^{(\bsig,\bp)}(\tilde \bu)=\tilde \bt\krog\tilde \bu$, we have $\prod_{i=1}^d\tilde\theta_i^{b_i}=r_{\bb}(F^{(\bsig,\bp)})$. Now, if $\rho(A)<1$, then Theorem 3.1 in \cite{mhpfpaper} implies that there exists a unique $\tilde\bu\in\S_{++}^{(\bsig,\bp)}$ such that $F^{(\bsig,\bp)}(\tilde\bu)=\tilde\bt\krog\tilde\bu$ for some $\tilde\bt\in\R^d_{++}$. If $\rho(A)=1$, then by Lemma \ref{characgraph}, we know that the $\bsig$-weak irreducibility of $T$ implies that the graph of the multi-homogeneous mapping $F^{(\bsig,\bp)}$ is strongly connected. Hence, Theorem 5.2 in \cite{mhpfpaper} implies the existence of $(\tilde\bt,\tilde \bu)\in\R^d_{++}\times\S_{++}^{(\bsig,\bp)}$ such that $F^{(\bsig,\bp)}(\tilde \bu)=\tilde \bt\krog\tilde \bu$. Furthermore, as $T$ is $\bsig$-weakly irreducible, by Lemma \ref{characgraph} we know that $DF^{(\bsig,\bp)}(\bx)$ is irreducible for every $\bx\in\kone_{++}^{\bsig}$. Hence, Theorem 6.2 in \cite{mhpfpaper} implies that $\tilde \bu$ is the unique vector in $\S_{++}^{(\bsig,\bp)}$ such that $F^{(\bsig,\bp)}(\tilde \bu)=\tilde \bt\krog\tilde \bu$. In any case, we have that there exists a unique $(\tilde\bt,\tilde \bu)\in\R^d_{++}\times\S_{++}^{(\bsig,\bp)}$ with $F^{(\bsig,\bp)}(\tilde \bu)=\tilde \bt\krog\tilde \bu$. Hence, Lemma \ref{connectmax} implies that the $(\bsig,\bp)$-eigenvector $\bu$ of \eqref{PFweak} can be chosen strictly positive. Finally, if $\bv\in \S_{++}^{(\bsig,\bp)}$ is a $(\bsig,\bp)$-eigenvector of $T$ such that $\bv\neq \bu$, then, by Lemma \ref{mhconnect}, there exists $\blam\in\R^d_{++}$ such that $F^{(\bsig,\bp)}(\bv)=\blam\krog\bv$ which is a contradiction as we have shown that $\bu$ is the unique vector in $\S_{++}^{(\bsig,\bp)}$ having this property. 
\item If $(\vartheta,\bx)$ is a $(\bsig,\bp)$-eigenpair of $T$ and $\bx\in\kone_{+,0}^{\bsig}\setminus\kone_{++}^{\bsig}$, then by Lemma \ref{mhconnect}, we have $F^{(\bsig,\bp)}_i(\bx)=\vartheta^{p_i'-1}\bx_i$ for all $i\in[d]$. Now, Theorem 6.2 in \cite{mhpfpaper} implies that, with $\gamma'=\frac{\gamma}{\gamma-1}=\sum_{i=1}^db_ip_i'$, we have $\vartheta=(\vartheta^{\gamma'-1})^{\gamma-1}<r_{\bb}(F^{(\bsig,\bp)})^{\gamma-1}=r^{(\bsig,\bp)}(T),$
where we have used Lemma \ref{specradlem} for the last equality.
\item Let $(\vartheta,\bx)$ be a $(\bsig,\bp)$-eigenpair of $T$ such that $\bx\in\kone_{+,0}^{\bsig}$. As $T$ is $\bsig$-strongly irreducible, Lemma \ref{characirr} implies that $F^{(\bsig,\bp)}$ satisfies the assumption of Lemma \ref{posvectirr}. In particular, as $F^{(\bsig,\bp)}_i(\bx)=\vartheta^{p_i'-1}\bx_i$ for all $i\in[d]$, Lemma \ref{posvectirr} implies that $\bx\in\S_{++}^{(\bsig,\bp)}$. As $\bsig$-strong irreducibility implies $\bsig$-weak irreducibility by Theorem \ref{irrthm}, we know by \eqref{uniquepos} that $\bu$ is the unique positive $(\bsig,\bp)$-eigenvector of $T$ in $\S_{++}^{(\bsig,\bp)}$ and thus $\bx=\bu$.
\end{enumerate}
\end{proof}

To prove Theorem \ref{PMthm}, we first introduce the following preliminary lemma:
\begin{lem}\label{Glem}
Let $G^{(\bsig,\bp)}$ be defined as in \eqref{defG}. Then, the following hold: 
\begin{enumerate}[(a)]
\item $G^{(\bsig,\bp)}$ is an order-preserving multi-homogeneous mapping. Furthermore, the homogeneity matrix of $G^{(\bsig,\bp)}$ is given by $B=(A+I)/2$ and $B^\top\bb=\rho(B)\bb$, where $A$ is the homogeneity matrix of $F^{(\bsig,\bp)}$.\label{111}
\item If $T$ is $\bsig$-strictly nonnegative, then $G^{(\bsig,\bp)}(\bx)\in\kone_{++}^{\bsig}$ for all $\bx\in\kone_{++}^{\bsig}$.
\item For every $\bu\in\kone_{++}^{\bsig}$, we have $F^{(\bsig,\bp)}(\bu)=\bt\krog\bu$ if and only if $G^{(\bsig,\bp)}(\bu)=\tilde\bt\krog\bu$ with $\tilde \theta_i^2=\theta_i$ for all $i\in[d]$.
\item It holds $r_{\bb}(G^{(\bsig,\bp)})^2=r_{\bb}(F^{(\bsig,\bp)})$.\label{222}
\item If $T$ is $\bsig$-weakly irreducible, then the Jacobian matrix $DG^{(\bsig,\bp)}(\bx)$ is primitive for every $\bx\in\kone_{++}^{\bsig}$.\label{333}
\end{enumerate}
\end{lem}
\begin{proof}
\eqref{111}-\eqref{222} follow by a straightforward calculation. For \eqref{333}, note that $$DG^{(\bsig,\bp)}(\bx) = \frac{1}{2}\diag(G^{(\bsig,\bp)}(\bx))^{-1/2}\big(\diag(F^{(\bsig,\bp)}(\bx))+\diag(x)DF^{(\bsig,\bp)}(\bx) \big).$$ As $T$ is $\bsig$-weakly irreducible, $DF^{(\bsig,\bp)}(\bx)$ is irreducible by Lemma \ref{characgraph}. It follows that $DG^{(\bsig,\bp)}(\bx)$ is primitive.
\end{proof}

\begin{proof}[Proof of Theorem \ref{PMthm}]
We begin with general observations: As $\bu\in\S_{++}^{(\bsig,\bp)}$ is a positive  $(\bsig,\bp)$-eigenvector of $T$, we know by \eqref{finalCW} that its corresponding $(\bsig,\bp)$-eigenvalue is $\lambda=r^{(\bsig,\bp)}(T)$. Furthermore, Lemmas \ref{mhconnect} and \ref{Glem} imply that $F^{(\bsig,\bp)}_i(\bu)=\lambda^{p_i'-1}\bu_i$ and $G^{(\bp,\bsig)}_i(\bu)=\lambda^{(p_i'-1)/2}\bu_i$ for all $i\in[d]$. Lemmas \ref{specradlem} and \ref{Glem} imply that $\lambda=r_{\bb}(F^{(\bsig,\bp)})^{\gamma-1}=r_{\bb}(G^{(\bsig,\bp)})^{2(\gamma-1)}$.
To show \eqref{genPM}, let $\omega\in\{\xi,\zeta\}$. Then \eqref{monoseq} follow from Lemma 7.4 in \cite{mhpfpaper}. Now, suppose that $\epsilon >0$ and $\widehat\omega_{k} -\widecheck\omega_{k} <\epsilon$. Then, \eqref{stopcrit} is obtained by subtracting $(\widehat\omega_{k} +\widecheck\omega_{k} )/2$ from $\widehat\omega_{k} \leq \lambda \leq \widecheck\omega_{k}$. Finally, with Lemma \ref{Glem}, \eqref{333}, we have that \eqref{FPM} and \eqref{GPM} both follow from Theorem 7.1 in \cite{mhpfpaper}.
\end{proof}
\section*{Acknowledgments}\addcontentsline{toc}{section}{Acknowledgments}
The authors are grateful to Shmuel Friedland and Lek-Heng Lim,  for a number of insightful discussions and for pointing out relevant references, and to the anonymous referees for the several very useful comments they pointed out which remarkably improved the presentation of this work.


\end{document}